\documentclass[12pt]{amsart}
\usepackage{amssymb,amsmath,amsfonts,latexsym}
\usepackage{bm}
\usepackage[all,cmtip]{xy}
\usepackage{amscd}
\usepackage{multirow,bigdelim}
\usepackage{graphicx}
\usepackage{hyperref}
\newcommand{\newsection}[1]{\setcounter{equation}{0} \section{#1}}
\newcommand{\bea}{\begin{eqnarray}}
\newcommand{\eea}{\end{eqnarray}}

\newcommand{\al}{\alpha}

\newcommand{\clb}{\mathcal{B}}
\newcommand{\clc}{\mathcal{C}}
\newcommand{\cld}{\mathcal{D}}
\newcommand{\cle}{\mathcal{E}}
\newcommand{\clf}{\mathcal{F}}

\newcommand{\clh}{\mathcal{H}}
\newcommand{\clk}{\mathcal{K}}
\newcommand{\cll}{\mathcal{L}}
\newcommand{\clm}{\mathcal{M}}
\newcommand{\cln}{\mathcal{N}}

\newcommand{\clp}{\mathcal{P}}

\newcommand{\clw}{\mathcal{W}}

\newcommand{\bz}{\bm{z}}
\newcommand{\bw}{\bm{w}}

\newcommand{\m}{\mathcal}
\newcommand{\z}{\mathcal B}

\newcommand{\vp}{\varphi}

\newcommand{\B}{\mathbb{B}}
\newcommand{\C}{\mathbb{C}}
\newcommand{\D}{\mathbb{D}}
\newcommand{\N}{\mathbb{N}}
\newcommand{\Z}{\mathbb{Z}}

\newcommand{\raro}{\rightarrow}

\def\textmatrix#1&#2\\#3&#4\\{\bigl({#1 \atop #3}\ {#2 \atop #4}\bigr)}
\def\dispmatrix#1&#2\\#3&#4\\{\left({#1 \atop #3}\ {#2 \atop #4}\right)}
\newcommand{\be}{\begin{equation}}
\newcommand{\ee}{\end{equation}}
\newcommand{\ben}{\begin{eqnarray*}}
\newcommand{\een}{\end{eqnarray*}}

\newcommand{\bi}{\begin{itemize}}
\newcommand{\ei}{\end{itemize}}

\newcommand\la{{\langle }}
\newcommand\ra{{\rangle}}

\theoremstyle{plain}

\newtheorem{Theorem}{\sc Theorem}[section]
\newtheorem{Lemma}[Theorem]{\sc Lemma}
\newtheorem{Proposition}[Theorem]{\sc Proposition}
\newtheorem{Corollary}[Theorem]{\sc Corollary}
\newtheorem{Definition}[Theorem]{\sc Definition}
\newtheorem{Example}[Theorem]{\sc Example}
\newtheorem{Remark}[Theorem]{\sc Remark}
\newtheorem{Note}[Theorem]{\sc Note}
\newtheorem{Question}{\sc Question}
\newtheorem{ass}[Theorem]{\sc Assumption}
\newcommand{\bt}{\begin{Theorem}}
\def\beginlem{\begin{Lemma}}
\def\beginprop{\begin{Proposition}}
\def\begincor{\begin{Corollary}}
\def\begindef{\begin{Definition}}
\def\beginexamp{\begin{Example}}
\def\beginrem{\begin{Remark}}
\def\beginq{\begin{Question}}
\def\beginass{\begin{ass}}
\def\beginnote{\begin{Note}}
\newcommand{\et}{\end{Theorem}}
\def\endlem{\end{Lemma}}
\def\endprop{\end{Proposition}}
\def\endcor{\end{Corollary}}
\def\enddef{\end{Definition}}
\def\endexamp{\end{Example}}
\def\endrem{\end{Remark}}
\def\endq{\end{Question}}
\def\endass{\end{ass}}
\def\endnote{\end{Note}}

\begin{document}

\title{Commuting row contractions with polynomial characteristic functions}

\dedicatory{To the memory of Ciprian Foias}

\author[Bhattacharjee]{Monojit Bhattacharjee}
\address{Department of Mathematics, Indian Institute of Technology Bombay, Powai, Mumbai, 400076, India}
\email{mono@math.iitb.ac.in, monojit.hcu@gmail.com}

\author[Haria] {Kalpesh J. Haria}
\address{School of Basic Sciences, Indian Institute of Technology Mandi, Mandi, 175005, Himachal Pradesh, India}
\email{kalpesh@iitmandi.ac.in, hikalpesh.haria@gmail.com}

\author[Sarkar]{Jaydeb Sarkar}
\address{Indian Statistical Institute, Statistics and Mathematics Unit, 8th Mile, Mysore Road, Bangalore, 560059, India}
\email{jay@isibang.ac.in, jaydeb@gmail.com}

\makeatletter
\@namedef{subjclassname@2020}{%
	\textup{2020} Mathematics Subject Classification}
\makeatother

\begin{abstract}
A characteristic function is a special operator-valued analytic function defined on the open unit ball of $\mathbb{C}^n$ associated with an $n$-tuple of commuting row contraction on some Hilbert space. In this paper, we continue our study of the representations of $n$-tuples of commuting row contractions on Hilbert spaces, which have polynomial characteristic functions. Gleason's problem plays an important role in the representations of row contractions. We further complement the representations of our row contractions by proving theorems concerning factorizations of characteristic functions. We also emphasize the importance and the role of the noncommutative operator theory and noncommutative varieties to the classification problem of polynomial characteristic functions.
\end{abstract}

\maketitle

\newsection{Introduction}

Identifying and then computing a complete unitary invariant of (tuples of) bounded linear operators on Hilbert spaces is one of the central objects in operator theory. From this point of view, the notion of characteristic function of contractions on Hilbert spaces stands out in its breadth of applications in function theory and operator theory.

Let $T=(T_1, \ldots, T_n)$ be an $n$-tuple of commuting operators on a Hilbert space $\clh$, and let $T$ be a row contraction (that is, $\sum_{i=1}^n T_i T_i^* \leq I_{\clh}$). The characteristic function of $T$ is the $\clb(\cld_T, \cld_{T^*})$-valued analytic function
\[
\theta_T(z_1, \ldots, z_n) = [-T + D_{T^*}\big(I_{\clh} - \sum_{i=1}^n z_i T_i^*\big)^{-1}Z D_T]|_{\cld_T},
\]
for all $(z_1, \ldots, z_n) \in \B^n$, where $\B^n$ denotes the open unit ball in $\C^n$, $D_T = (I - T^*T)^{\frac{1}{2}}$ and $\cld_{T} = \overline{\mbox{ran}} D_T$ (see Section \ref{sect-2} for more details).

In particular, if $n=1$, then the above definition of $\theta_T$ becomes the well known and classical Sz.-Nagy and Foias characteristic function of the single contraction $T$ \cite{NF}. In this case, clearly, $\theta_T$ admits a power series expansion on the disc $\D = \{z \in \C: |z|<1\}$. This, of course, immediately raises the natural question of the relationship between the class of polynomial characteristic functions and the structure of corresponding contractions. To some extent, the work of Foias and the third author \cite{SF12} gives a satisfactory answer to this question. For instance: The characteristic function $\theta_T$ of a completely
nonunitary contraction $T$ on a separable, infinite dimensional,
complex Hilbert space $\clh$ is a polynomial if and only if there
exist three closed subspaces $\clh_1, \clh_0, \clh_{-1}$ of $\clh$
with $\clh = \clh_1 \oplus \clh_0 \oplus \clh_{-1}$, a pure isometry
$S$ on $\clh_1$, a nilpotent $N$ on $\clh_0$, and a pure
co-isometry $C$ on $\clh_{-1}$, such that $T$ admits the following matrix representation
\[
T = \begin{bmatrix} S & * & *\\0 & N & *\\0& 0& C
\end{bmatrix}.
\]
Moreover, the dimension of $\ker S^*$ and dimension of $\ker C$ are unitary invariants of $T$ and that $N$, up to a quasi-similarity, is uniquely determined by $T$ (see \cite[Sections 4 and 5]{SF12}). In the follow-up paper, Foias, Pearcy and the third author \cite{FPS17} proved the following analytic result: If $\theta_T$ is a polynomial of degree $m$, then there exist a Hilbert space $\clm$, a nilpotent operator $N$ of order $m$, a coisometry $V_1 \in \clb(\cld_{N^*} \oplus \clm, \cld_{T^*})$, and an isometry $V_2  \in \clb(\cld_T, \cld_N \oplus \clm)$, such that
\[
\theta_T = V_1 \begin{bmatrix} \theta_N & 0 \\ 0 & I_{\clm} \end{bmatrix} V_2.
\]
On the other hand, the approach of \cite{SF12} was continued and extended to $n$-tuples of noncommuting row contractions setting by Popescu in \cite{Po13}. Also, the results of \cite{FPS17} were further extended to Popescu's noncommutative setting in \cite{HMS17}.

It is worthwhile to note that Popescu (see \cite{P10} and other references therein) first recognized that the notion of characteristic functions, a special class of multi-analytic operators \cite{Po06b}, plays a central role in multivariable operator theory and noncommutative function theory. Moreover, his approach to noncommutative verities links up with the noncommutative operator theory and commutative operator theory (see
\cite{Po06a, P10} and Section \ref{sect-5}).

This paper aims to complete the classification problem of contractions, which admits polynomial characteristic functions. More precisely, we aim to classify $n$-tuples of commuting row contractions, which admits polynomial characteristic functions.

The question of the structure of $n$-tuples of commuting row contractions is important in its own right. However, on the other hand, Popescu's approach to noncommutative verities unifies many analytic and geometric questions concerning $n$-tuples of commuting row contractions. From this point of view, it is also necessary to examine the noncommutative operator theoretic technique and the classifications of noncommuting row contractions admitting polynomial characteristic functions to our classification problem of tuples of commuting row contractions. As we will see, some of the present techniques and results are similar to the one variable case and the noncommutative case. However, commutativity property (a constrained property, as identified by Popescu in \cite{Po06a, P10}) brings out more intrinsic function theoretic features to the classification problem. Indeed, natural and satisfactory versions of the classification problem (for instance, see  Theorem \ref{Drury Arveson Shift}) are related to the notion of Gleason’s problem (see Definition \ref{regularity definition}). In this context, we also refer to the paragraph following Theorem \ref{rm1}.

The remaining part of the paper is organized as follows: In Section \ref{sect-2}, we briefly outline a few key facts of Drury-Arveson space, $n$-tuples of commuting row contractions and characteristic functions of commuting row contractions. Section \ref{sect-3} deals with the structure of $n$-tuples of commuting row contractions, which admits polynomial characteristic functions. Section \ref{sect-4} is devoted to the study of factorizations of characteristic functions of $n$-tuples of noncommutative row contractions. In Section \ref{sect-5}, we continue our discussion of factorizations of characteristic functions in the setting of noncommutative varieties. In Section \ref{sect-6} we discuss some unitary invariants of $n$-tuples of commuting row contractions, which admits polynomial characteristic functions. The final section is devoted to an example to justify the regularity assumption on commuting tuples of row contractions. 

From the multivariable operator theory point of view, this is a sequel to the papers \cite{SF12} and \cite{FPS17} by Foias, and Foias and Pearcy, respectively, and the third author.

\newsection{Preliminaries}\label{sect-2}

In this section, we recall basic definitions and notations used in the rest of the paper. Throughout the paper, Hilbert spaces will be denoted by $\clh_1$, $\clh_2$, $\cle$, $\cle_*$, etc. The set of all bounded linear operators from $\clh_1$ to $\clh_2$ will be denoted by $\clb(\clh_1, \clh_2)$. When $\clh_1 = \clh_2$, one writes simply $\clb(\clh_1)$ instead of $\clb(\clh_1, \clh_1)$. Now let $\{T_1, \ldots, T_n\} \subseteq \clb(\clh)$. We say that $T = (T_1, \ldots, T_n)$ is a \textit{row contraction} (or \textit{spherical contraction}) if the row operator $T : \clh^n \raro \clh$, defined by
\[
T(h_1,\ldots, h_n) = \sum_{i=1}^n T_ ih_i \quad \quad (h_1, \ldots, h_n \in \clh),
\]
is a contraction. It is clear that $T$ is a row contraction if and only if $\sum\limits_{i=1}^n \|T_i h_i\|^2 \leq \|h\|^2 $ for all $h \in \clh$, or equivalently $\sum_{i=1}^n T_i T^*_i \leq I_{\clh}$. A row contraction  $T$ is said to be \textit{commuting row contraction} if $T_i T_j = T_j T_i$ for $i, j= 1,\ldots, n$.

A typical example of a commuting row contraction is the $n$-tuple of multiplication operator $(M_{z_1}, \ldots, M_{z_n})$ on the \textit{Drury-Arveson space} $H^2_n$, where $H^2_n$ is the reproducing kernel Hilbert space corresponding to the kernel
\[
k(\bz, \bw) = (1 - \sum\limits_{i=1}^n z_i \bar{w}_i)^{-1} \quad \quad (\bz, \bw \in \mathbb{B}^n).
\]
Here $\mathbb{B}^n$ denotes the open unit ball in $\mathbb{C}^n$ and $\bz$ (and $\bw$ etc.) denotes an element in $\mathbb{C}^n$, that is, $\bz = (z_1, \ldots, z_n) \in \mathbb{C}^n$. Then
\[
H^2_n = \{f = \sum_{\alpha \in \Z_+^n} a_{\alpha} \bz^{\alpha}: a_{\alpha} \in \mathbb{C} \mbox{~and~} \|f\|^2 := \sum_{\alpha \in \Z_+^n} \frac{|a_{\alpha}|^2}{\gamma_{\alpha}} < \infty\},
\]
where $ \Z_+ = \{0, 1, 2, \ldots\} $, $\alpha = (\alpha_1, \ldots, \alpha_n) \in \Z_+^n$ and
\[
\gamma_{\alpha} := \frac{|\alpha|!}{\alpha!} = \frac{(\sum_{i=1}^n \alpha_i)!}{\alpha_1! \cdots \alpha_n!},
\]
is the multinomial coefficient. The $\cle$-valued Drury-Arveson space will be denoted by $H^2_n(\cle)$. In this case, the representation of $H^2_n(\cle)$ is the same as $H^2_n$ above but replacing $a_{\alpha} \in \mathbb{C}$ with $a_{\alpha} \in \cle$ and $|a_{\alpha}|$ with $\|a_{\alpha}\|_{\cle}$. Now identifying $H^2_n(\cle)$ with the Hilbert space tensor product $H^2_n \otimes \cle$ (via $\bz^{\alpha} \eta \mapsto \bz^{\alpha} \otimes \eta$, $\alpha \in \Z_+^n$ and $\eta \in \cle$), we see that $(M_{z_1}, \ldots, M_{z_n})$ on $H^2_n(\cle)$ and $(M_{z_1} \otimes I_{\cle}, \ldots, M_{z_n} \otimes I_{\cle})$ on $H^2_n \otimes \cle$ are unitarily equivalent. We shall frequently make use of this identification. Given a commuting tuple $M = (M_1, \ldots, M_n)$ on a Hilbert space $\clh$, we often say that $M$ is a \textit{Drury-Arveson shift} if there exists a Hilbert space $\clw$ such that $M$ and $(M_{z_1}, \ldots, M_{z_n})$ on $H^2_n (\clw)$ are unitarily equivalent.

Also recall that a holomorphic function $\vp : \mathbb{B}^n \raro \clb(\cle, \cle_*)$ is said to be a (Drury-Arveson) \textit{multiplier} if
\[
\vp H^2_n(\cle) \subseteq H^2_n(\cle_*).
\]
In this case, by virtue of the closed graph theorem, it follows that the multiplication operator $M_{\vp} : H^2_n(\cle) \raro H^2_n(\cle_*)$ (where $M_{\vp}f = \vp f$ for all $f \in H^2_n(\cle)$) is a bounded linear operator. The set of all multipliers will be denoted by $\clm(\cle, \cle_*)$. It also follows that $\clm(\cle, \cle_*)$ is a Banach space relative to the operator norm
\[
\|\vp\| := \|M_{\vp}\|_{\clb(H^2_n(\cle), H^2_n(\cle_*))} \quad \quad (\vp \in \clm(\cle, \cle_*)).
\]
Now let $T = (T_1, \ldots, T_n)$ be a row contraction on $\clh$. The \textit{defect operators} and \textit{defect spaces} of $T$ are given by
\[
D_{T} = (I- T^* T)^\frac{1}{2} \in \z(\clh^n) \quad \mbox{and} \quad D_{{T}^*} = (I- TT^*)^\frac{1}{2} \in \z(\m H),
\]
and
\[
\m D_{T} = \overline{\mbox{ran}} D_T \subseteq \clh^n \quad \mbox{and} \quad \m D_{T^*} = \overline{\mbox{ran}} D_{T^*} \subseteq \m H,
\]
respectively. For any commuting row contraction $T = (T_1, \ldots, T_n)$ on $\clh$, the \textit{characteristic function} of $T$ is a $\clb(\cld_T, \cld_{T^*})$-valued analytic function $\theta_T: \mathbb{B}^n \raro \clb(\cld_T, \cld_{T^*})$ defined by
\begin{equation}\label{eq-ch fn commutative}
\theta_T(\bz) = \Big(-T + D_{T^*}\big(I_{\clh} - Z T^*\big)^{-1}Z D_T\Big)|_{\cld_T} \quad \quad (\bz \in \B^n),
\end{equation}
where $Z = (z_1 I_{\clh}, \ldots, z_n I_{\clh})$ is a row operator on $\clh$ and so $ZT^* = \sum\limits_{i=1}^n z_i T_i^*$ for all $\bz \in \B^n$. Also we define $T^{\alpha} = T_1^{\alpha_1} \cdots T_n^{\alpha_n}$ and $T^{*\alpha} = T_1^{*\alpha_1} \cdots T_n^{*\alpha_n}$ for all $\alpha = (\alpha_1, \ldots, \alpha_n) \in \mathbb Z_+^n$, and $P_j: \clh^n \to \m H$ by
\[
P_j (h_1, \dots, h_n) = h_j \quad \quad (h_1, \ldots, h_n \in \clh).
\]
Then
\[
\theta_T(\bz) = \Big(-T + D_{T^*} \sum_{\substack{\alpha \in \Z_+^n\\ j=1}}^n \gamma_{\alpha} T^{*\al} \bz^{\al + e_j} P_j D_T\Big)|_{\cld_T}.
\]
If we define $\theta_{T, \alpha}$, the coefficient of $\bz^{\alpha}$, $\alpha \in\Z_+^n$, in the Taylor series expansion of $\theta_T$ as
\[
\theta_{T, \alpha} =
\begin{cases} - T|_{\cld_T} & \mbox{if}~ \alpha = 0
\\
\displaystyle\sum_{j=1}^n \gamma_{\alpha - e_j} D_{T^*} T^{*(\al - e_j)}  P_j D_T|_{\cld_T} &  \mbox{if}~ \alpha \neq 0, \end{cases}
\]
then $\theta_T(\bz) = \sum_{|\alpha| \geq 0} \theta_{T, \alpha} \bz^{\al}$, $\bz \in \B^n$. In what follows, we adopt the standard convention that
\[
\gamma_{\alpha - e_j} = 0 \quad \mbox{and} \quad T^{*(\al - e_j)} = I \quad \quad (\alpha \in \Z_+^n, \alpha_j = 0).
\]
It is now natural to define polynomial characteristic functions. Let $T$ be an $n$-tuple of commuting row contraction on $\clh$ and let $m$ be a natural number. We say that the characteristic function $\theta_T$ is a \textit{polynomial of degree $m$} if
\[
\theta_{T, \alpha} \neq 0,
\]
for some $|\alpha| = m$ and $\theta_{T, \beta} = 0$ for all $|\beta| > m$. If $\theta_T(\bz) \equiv -T|_{\cld_T}$, $\bz \in \mathbb{B}^n$, then we say that $\theta_T$ is a \textit{polynomial of degree zero}. Throughout this paper, we make the convention that the degree of the \textit{zero polynomial} is zero.

A commuting tuple $N = (N_1, \ldots, N_n)$ on $\clh$ is said to be \textit{nilpotent of order $m ( > 1)$} if
\[
N^{\alpha} = 0 \quad \mbox{and} \quad N^{\beta} \neq 0,
\]
for all  $\al$  in $\Z_+^n$ with $|\alpha| = m$  and for some $\beta$ in $\Z_+^n$ such that $|\alpha| - |\beta| = 1$. For a commuting row contraction $T = (T_1, \ldots, T_n)$ on $\clh$ we define
\begin{equation}\label{cnc}
\clh_c := \Big\{ h\in \clh : \sum_{|\al|=k} \|{T^*}^\al h \|^2 = \|h\|^2 \quad \text{for all} \quad k \in \Z_+ \Big\} .
\end{equation}
Clearly, $\clh_c$ is a closed and joint $(T_1^*, \ldots, T_n^*)$-invariant subspace. Moreover, $\clh_c$ is \textit{maximal}, that is, $\clh_c$ is the largest closed subspace of $\clh$ on which $T^* : \clh \raro \clh^n$ acts isometrically. The row contraction $T$ is said to be a \textit{completely non-coisometric} (c.n.c) row contraction if $\clh_c =\{0\}$. The row contraction $T$ is said to be \textit{pure} if
\[
\lim_{k \to \infty} \sum_{\substack{\al \in \Z_+^n \\ |\al|=k}} \|{T^*}^\al h \|^2 = 0 \quad \quad (h\in \clh).
\]
As an example, we note that the multiplication tuple $(M_{z_1}, \ldots, M_{z_n})$ on a vector-valued Drury-Arveson space $H^2_n(\cle)$ is a pure row contraction.

Finally, we recall that a pair of commuting $n$-tuples of row contractions $(T_1,\ldots, T_n)$ and $(T^{'}_1,\ldots, T^{'}_n)$ are said to be \textit{unitary equivalent} if there exists a unitary $U : \clh \raro \clh^{'}$ such that $T_i=UT^{'}_i U^*$ for all $i=1, \ldots, n$.

\newsection{Polynomial Characteristic Functions}\label{sect-3}

This section presents the representations of $n$-tuples of commuting row contractions, which admits polynomial characteristic functions. Gleason's problem plays a crucial role in our consideration. We begin with the following key lemma.

\begin{Lemma}\label{evaluation of adjoint T_j}
Let $T = (T_1, \ldots, T_n)$ be a commuting row contraction on a Hilbert space $\clh$. Suppose $\theta_T$ is a polynomial of degree $m$. If $\alpha \in \Z_+^n$ and $|\alpha| \geq m+1$, then
\[
T_i^* (T^{\al} D_{T^*}) =
 \frac{\al_i}{|\al|} (T^{\al - e_i} D_{T^*}),
\]
for all $i \in \{1, \ldots, n\}$.
\end{Lemma}
\begin{proof} Fix $i \in \{1, \ldots, n\}$. For each $|\alpha| \geq  m+1$, since $\theta_{T, \alpha}^* = 0$, it follows that
\[
D_T^2 \sum_{j=1}^n \gamma_{\al - e_j} P_j^* T^{\al - e_j} D_{T^*} = 0.
\]
Note that $P_j^*: \m H \raro \clh^n$ is given by
\[
P_j^*(h) = (0, \ldots, 0, \underbrace{h}_{\text{j-th position}}, 0, \ldots, 0),
\]
for all $h \in \clh$. Therefore, using matrix representation of the operator $D_T^2$, we have
\[
\begin{bmatrix}
I-T_1^*T_1 & -T_1^*T_2  \cdots & -T_1^*T_n  \cr
-T_2^*T_1  & I-T_2^*T_2 \cdots & -T_2^*T_n  \cr
\vdots     & \vdots            & \vdots     \cr
-T_n^*T_1  & -T_n^*T_2  \cdots & I-T_n^*T_n \cr
\end{bmatrix} \begin{bmatrix}
\delta_{\al_1} \cr \delta_{\al_2} \cr \vdots \cr \delta_{\al_n}
\end{bmatrix}= \begin{bmatrix}
0 \cr 0 \cr \vdots \cr 0
\end{bmatrix}
\]
where
\[
\delta_{\al_j} =
\gamma_{\alpha-e_j} T_j^{\al_j-1} T_1^{\al_1} \cdots T_n^{\al_n}D_{T^*},
\]
for all $\al =(\al_1,\ldots,\al_n) \in \Z_+^n$ with $ \al \geq m+1 $ and $j=1,\ldots, n$. From the above identity, we have
\[
\sum_{\substack{j =1 \\ j\neq i}}^n - T_i^*T_j \delta_{\al_j} + (I - T_i^* T_i) \delta_{\al_i} = 0,
\]
and hence
\begin{align*}
\delta_{\al_i} =  T_i^*\big( \sum_{j=1}^n T_j \delta_{\al_j}\big).
\end{align*}
Replacing $\delta_{\al_j} =
\gamma_{\alpha-e_j} T_j^{\al_j-1} T_1^{\al_1} \cdots T_n^{\al_n}D_{T^*}$ in the above identity, we get
\begin{align*}
\gamma_{\alpha - e_i} T_i^{\al_i - 1} T_1^{\al_1} \cdots T_n^{\al_n}D_{T^*}
& = T_i^*\big( \sum_{j = 1}^n T_j  (\gamma_{\alpha - e_j} T_j^{\al_j - 1} T_1^{\al_1} \cdots T_n^{\al_n}D_{T^*} )\big)
\\
& =\big[\sum_{j=1}^n \gamma_{\alpha- e_j} \big] (T_i^* T_1^{\al_1} \cdots T_n^{\al_n}D_{T^*} )
\\
&= \big[\sum_{j=1}^n \gamma_{\alpha - e_j} \big] (T_i^* (T^\alpha D_{T^*})).
\end{align*}
Since $ |\alpha| \geq m+ 1 $, $ \alpha_k \geq 1 $ for some $ k\in \{1, \ldots, n\}$. Therefore $ \gamma_{\alpha-e_k} \neq 0$  for some $ k\in \{1, \ldots, n\} $.
Then
\begin{align*}
T_i^* (T^\alpha D_{T^*} )& = \dfrac{\gamma_{\alpha-e_i} }{\big[\displaystyle \sum_{j=1}^n \gamma_{\alpha - e_j}\big]} T_i^{\al_i - 1} T_1^{\al_1} \cdots T_n^{\al_n}D_{T^*}
\\
&=  \dfrac{\gamma_{\alpha-e_i} }{\big[\displaystyle\sum_{j=1}^n \gamma_{\alpha - e_j}\big]} (T^{\al - e_i} D_{T^*}).
\end{align*}
Finally, since
\[
\dfrac{\gamma_{\alpha-e_i} }{\big[ \displaystyle\sum_{j=1}^n \gamma_{\alpha-e_j} \big] }  =  \frac{\al_i}{|\al|},
\]
it follows that $T_i^* (T^{\al} D_{T^*}) =
\frac{\al_i}{|\al|} (T^{\al - e_i} D_{T^*})$.
\end{proof}

\begin{Lemma}\label{orthogonal}
Let $T$ be an $n$-tuple of commuting row contraction on $\clh$. If $\theta_T$ is a polynomial of degree $m$, then
\[
T^{\al}\cld_{T^*} \perp T^{\beta}\cld_{T^*},
\]
for all $\al, \beta \in \Z_+^n$, $\al \neq \beta$ and $ |\al|, |\beta| \geq m$.
\end{Lemma}
\begin{proof} If $\gamma \in \Z_+^n$, $|\gamma| \geq m$ and $i=1, \ldots, n$, then by Lemma \ref{evaluation of adjoint T_j}, we have
\begin{equation}\label{eq-lemma 3.2}
T^{\gamma} D_{T^*} = \frac{|\gamma| + 1}{\gamma_{i}+1} T_i^* T_i T^{\gamma} D_{T^*}.
\end{equation}
Now we fix $\al, \beta \in \Z_+^n$ such that $\al \neq \beta$ and $ |\al|, |\beta| \geq m$. Since $\al \neq \beta$, $\al_j \neq \beta_j$ for some $j \in \{1, \ldots,n\}$. Without loss of generality, we assume that $\al_j < \beta_j$. Fix an integer $k \in \{1, \ldots, n\}$ such that $k \neq j$. By \eqref{eq-lemma 3.2}, we have
\[
T^{\beta}D_{T^*} = c_k T_k^*T_k T^
{\beta}D_{T^*},
\]
where
\[
c_k = \frac{|\beta| + 1}{\beta_k +1}.
\]
By repeated applications of \eqref{eq-lemma 3.2}, we have
\[
T^{\beta} D_{T^*} = \big(c_k \cdots c_{k+m+1}\big) T_k^{* m+1}T_k^{m+1} T^{\beta}D_{T^*},
\]
for some positive scalars $c_k, \ldots, c_{k+m+1}$. Hence for $h_1$ and $h_2$ in $\clh$, we have
\[
\la T^{\al}D_{T^*}h_1, T^{\beta}D_{T^*}h_2 \ra = \big(c_k \cdots c_{k+m+1}\big) \la T_k^{m+1} T^{\al}D_{T^*}h_1, T_k^{m+1} T^{\beta}D_{T^*}h_2 \ra,
\]
where, on the other hand
\[
\begin{split}
\la T_k^{m+1} T^{\al}D_{T^*}h_1, T_k^{m+1} T^{\beta}D_{T^*}h_2 \ra & = \la T_j^{\al_j} T_k^{m+1} \Big(\mathop\Pi_{i \neq j} T_i^{\al_i} D_{T^*} \Big) h_1, T_j^{\beta_j} T_k^{m+1} \Big(\mathop\Pi_{i \neq j} T_i^{\beta_i} D_{T^*}\Big) h_2 \ra
\\
& = \la T_j^{*\beta_j} T_j^{\al_j} T_k^{m+1} \Big(\mathop\Pi_{i \neq j} T_i^{\al_i} D_{T^*} \Big) h_1, T_k^{m+1} \Big(\mathop\Pi_{i \neq j} T_i^{\beta_i} D_{T^*}\Big) h_2 \ra.
\end{split}
\]
But
\[
\begin{split}
T_j^{*\beta_j} T_j^{\al_j} T_k^{m+1} \Big(\mathop\Pi_{i \neq j} T_i^{\al_i} D_{T^*} \Big) & = T_j^{*(\beta_j - 1)} (T_j^* T_j) \Big(T_j^{\al_j - 1} T_k^{m+1} \mathop\Pi_{i \neq j} T_i^{\al_i} D_{T^*} \Big)
\\
& = c T_j^{*(\beta_j - 1)} \Big(T_j^{\al_j - 1} T_k^{m+1} \mathop\Pi_{i \neq j} T_i^{\al_i} D_{T^*} \Big),
\end{split}
\]
for some positive scalar $c$, which follows from Lemma \ref{evaluation of adjoint T_j}. By setting $\tilde{c} = c c_k \cdots c_{k+m+1}$, it follows that
\[
\la T^{\al}D_{T^*}h_1, T^{\beta}D_{T^*}h_2 \ra = \tilde{c} \la T_j^{*(\beta_j - 1)} \Big(T_j^{\al_j - 1} T_k^{m+1} \mathop\Pi_{i \neq j} T_i^{\al_i} D_{T^*} \Big) h_1, T_k^{m+1} \Big(\mathop\Pi_{i \neq j} T_i^{\beta_i} D_{T^*}\Big) h_2 \ra.
\]
Since $ \beta_j > \al_j $, applying again Lemma \ref{evaluation of adjoint T_j} (possibly finitely many times), we get a constant $\hat{c}$ such that
\[
T_j^{*(\beta_j - 1)} \Big(T_j^{\al_j - 1} T_k^{m+1} \mathop\Pi_{i \neq j} T_i^{\al_i} D_{T^*} \Big) = \hat{c} T_j^{*(\beta_j - \alpha_j)} \Big(T_k^{m+1} \mathop\Pi_{i \neq j} T_i^{\al_i} D_{T^*} \Big),
\]
and hence
\[
\la T^{\al}D_{T^*}h_1, T^{\beta}D_{T^*}h_2 \ra = \tilde{c} \hat{c} \la T_j^* \Big(T_k^{m+1} \mathop\Pi_{i \neq j} T_i^{\al_i} D_{T^*} \Big) h_1, T_j^{\beta_j - \alpha_j - 1} T_k^{m+1} \Big(\mathop\Pi_{i \neq j} T_i^{\beta_i} D_{T^*}\Big) h_2 \ra.
\]
But once again, by Lemma \ref{evaluation of adjoint T_j}, it follows that
\[
T_j^* \Big(T_k^{m+1} \mathop\Pi_{i \neq j} T_i^{\al_i} D_{T^*} \Big) =0.
\]
This implies that $\la T^{\al}D_{T^*}h_1, T^{\beta}D_{T^*}h_2 \ra = 0$ and completes the proof of the lemma.
\end{proof}

Now let $T$ be an $n$-tuple of commuting row contraction on $\clh$ such that the characteristic function $\theta_T$ is a polynomial of degree $m$. Set
\[
\clm = \overline{\text{span}} \{ T^{\al}D_{T^*}h : h \in \clh,
|\al| \geq m, \al \in \Z_+^n \},
\]
and
\[
\cln =  \overline{\text{span}}
\{ T^{\al}D_{T^*}h : h \in \clh, |\al|=m, \al \in \Z_+^n \}.
\]
Clearly, $\clm$ is a joint $T$-invariant subspace of $\clh$ and $\cln \subseteq \clm$. Define
\[
M_i:= T_i|_{\clm} \in \clb(\clm) \quad \quad (i=1, \ldots, n).
\]
Then $(M_1, \ldots, M_n)$ is a commuting row contraction on $\clm$. If $|\alpha| > m$, $\alpha \in \Z_+^n$, then Lemma \ref{evaluation of adjoint T_j} implies that
\[
M_i M_i^* (T^{\alpha} D_{T^*}) = M_i T_i^* T^{\alpha} D_{T^*} = \frac{\alpha_i}{|\alpha|} M_i T^{\alpha - e_i} D_{T^*} = \frac{\alpha_i}{|\alpha|} T^{\alpha} D_{T^*},
\]
for all $i=1, \ldots, n$, and hence
\[
\Big(\sum_{i=1}^n M_i M_i^*\Big)|_{\clm \ominus \cln} = I_{\clm \ominus \cln}.
\]
Moreover, if $\beta \in \Z_+^n$ and $|\beta| = m$, then, again, Lemma \ref{orthogonal} implies that
\[
T_i^* T^{\beta} \cld_{T^*} \perp T^{\gamma} D_{T^*},
\]
for all $i=1, \ldots, n$, and $\gamma  \in \Z_+^n$ and $|\gamma| \geq m$. This implies that $M_i^*|_{\cln} = 0$ for all $i = 1,\ldots,n$, and hence we find
\begin{equation}\label{eq-IM PN}
I_{\clm} - (M_1 M_1^* + \cdots + M_n M_n^*) = P_{\cln}.
\end{equation}
In particular, $\cln = \clm \ominus  \big(\displaystyle\sum_{i=1}^n M_i \m M\big)$. This also implies that the minimal closed joint $(M_1, \ldots, M_n)$-invariant subspace of $\clm$ containing $\cln$ is $\clm$ itself. Moreover, by virtue of Lemma \ref{evaluation of adjoint T_j}, it follows easily that $(M_1, \ldots, M_n)$ is a pure tuple. We summarize these observations as follows:

\begin{Theorem}\label{rm1}
Let $T = (T_1, \ldots, T_n)$ be a commuting row contraction on $\clh$. Assume that the characteristic function of $T$ is a polynomial of degree $m$. If
\[
\clm = \overline{\text{span}} \{ T^{\al}D_{T^*}h : h \in \clh,
|\al| \geq m, \al \in \Z_+^n \},
\]
and $M_i: = T_i|_{\clm}$ for all $i = 1,\ldots,n$, and
\[
\cln=  \overline{\text{span}}
\{ T^{\al}D_{T^*}h : h \in \clh, |\al|=m, \al \in \Z_+^n \},
\]
then $\clm$ is a joint closed invariant subspace for $T$ and the restriction tuple $M = (M_1, \ldots, M_n)$ is a commuting pure partial
isometry on $\clm$. Moreover
\[
\clm = \overline{\text{span}} \{ M^{\al} \cln: \al \in \Z_+^n \},
\]
and
\[
\cln = \clm \ominus  \big(\displaystyle\sum_{i=1}^n M_i \m M\big),
\]
and $\clm$ is the minimal closed joint $M$-invariant subspace of $\clm$ containing $\cln$.	
\end{Theorem}

A priori, the above result suggests that the $n$-tuple $M$ on $\clm$, up to unitary equivalence, is just the multiplication tuple $(M_{z_1}, \ldots, M_{z_n})$ on $H^2_n(\cln)$, the $\cln$-valued Drury-Arveson shift. It is also instructive to note that for $n=1$ case \cite{SF12} and for $n$-tuples of noncommutative operators \cite{Po13}, the operator $M$ on $\clm$ is indeed the multiplication operator or the tuple of creation operators on vector-valued Hardy space or the Fock space, respectively. However, for $n$-tuples of commuting row contractions, $n>1$, this is not true in general. This problem is connected to Gleason's property (also known as Gleason’s problem) of functions on the unit ball.

For the convenience of the reader, we recall  \textit{Gleason’s problem} in the Drury-Arveson space. Let $\bm{w} \in \B^n$ and let $f \in H^2_n$. If $f(\bm{w})=0$, then the Gleason problem says that \cite{AK} there exist $f_1, \ldots, f_n \in H^2_n$ such that
\[
f(\bz) = \sum_{i=1}^n (z_i - w_i) f_i(\bz) \quad \quad (\bz \in \B^n).
\]
Then, in view of the fact that $(M_z - W) (H^2_n)^n$ is a closed subspace of $H^2_n$  and
\[
\bigcap_{i=1}^n \ker (M_{z_i} - w_iI_{H^2_n})^* = \mathbb{C} k(\cdot, \bw),
\]
it follows that
\[
H^2_n = (M_z -W)(H^2_n)^n \mathop{+}^. \mathbb{C},
\]
for all $\bw \in \B^n$, where $\displaystyle\mathop{+}^.$ denotes the algebraic direct sum of subspaces. With this as motivation, we define regular tuples of operators \cite[Section 2]{EL18}.

\begin{Definition}\label{regularity definition}
We say that a tuple of commuting bounded linear operators $T = (T_1, \ldots, T_n)$ on a Hilbert space $\clh$ is regular if there exists $\epsilon > 0$ such that for any $\|\bz\|_{\C^n} < \epsilon$ the subspace $(T-Z)\clh^n$ is closed in $\clh$ and
\[
\clh = (T-Z)\clh^n \mathop{+}^. \Big(\clh \ominus \sum_{i=1}^n T_i \clh \Big).
\]
\end{Definition}

\begin{Theorem}\label{Drury Arveson Shift}
In the setting of Theorem \ref{rm1}, if, in addition, the $n$-tuple $M$ on $\clm$ is regular, then $M$ and the Drury-Arveson shift $(M_{z_1}, \ldots, M_{z_n})$ on $H^2_n(\m N)$ are unitary equivalent.
\end{Theorem}
\begin{proof}
By \eqref{eq-IM PN}, the tuple $M = (M_1, \ldots, M_n)$ on $\m M$ satisfies $I_{\clm} - M M^* = P_{\cln}$, and hence $M$ is a partial isometry and, in particular, $M^* M|_{\mbox{ran} M^*}: \text{ran} M^* \raro \text{ran} M^*$ is invertible. It follows that
\[
(M^* M)|_{\text{ran} M^*} = I_{\text{ran}M^*},
\]
and hence by \cite[Theorem 3.5]{EL18}, the map
\[
(Uf)(z) = \displaystyle\sum_{\alpha \in \Z_+^n} \gamma_{\alpha}\Big(P_{\m N} M^{*\alpha}f \Big) z^{\alpha},
\]
defines a unitary operator $U: \m M \to H^2_n(\m N)$ and satisfies $U M_i = M_{z_i} U$ for all $i=1, \ldots, n$.
\end{proof}

We continue with the setting of Theorem \ref{rm1}, and define
\[
\clk = \overline{\text{span}}\{T^{\al}D_{T^*}h : h \in \clh, \al \in \Z_+^n\},
\]
and
\[
\clh_{\text{nil}}= \clk \ominus \clm \quad \mbox{and} \quad N_i = P_{\clh_{\text{nil}}} T_i|_{\clh_{\text{nil}}},
\]
for all $i=1, \ldots, n$. Clearly, $\clh_{\text{nil}}$ is a semi-invariant subspace for $T$ and hence
\[
N^{\al} = P_{\clh_{\text{nil}}} T^{\al}|_{\clh_{\text{nil}}} \quad \quad (\al \in \Z_+^n).
\]
In particular, $N_i N_j = N_j N_i$ for all $i, j= 1, \ldots, n$, and
\[
\sum_{i=1}^n N_i N_i^* \leq P_{\clh_{\text{nil}}} \sum_{i=1}^n T_i T_i^*|_{\clh_{\text{nil}}} \leq I_{\clh_{\text{nil}}},
\]
that is, $N=(N_1,\ldots,N_n)$ is a commuting row contraction on $\clh_{\text{nil}}$. Clearly $\al \in \Z_+^n$ with $|\al|\geq m$ implies $T^{\al}\clk \subseteq \clm$ and hence $N^{\al}=0$. This shows that the commuting row contraction $N$ is a nilpotent tuple of order $\leq m$. Moreover, we have
\[
T_j|_{\clm \oplus \clh_{\text{nil}} } = \begin{bmatrix}
M_j & * \cr 0 & N_j
\end{bmatrix}: \clm \oplus \clh_{\text{nil}} \to \clm \oplus \clh_{\text{nil}}.
\]
Note now that $h \in \clh \ominus \clk$ if and only if $h \in \ker (D_{T^*}T^{* \al})$, or, equivalently, $h \in \ker (T^{\al} D^2_{T^*} T^{* \al})$ for all $\al \in \mathbb Z_+^n$. Note also that
\[
I - \sum_{|\al| = k} T^{\al} T^{*\al} = D_{T^*}^2 + (\sum_{|\beta| = 1} T^{\beta} D_{T^*}^2 T^{*\beta}) + \cdots + (\sum_{|\beta| = k-1} T^{\beta} D_{T^*}^2 T^{*\beta}),
\]
for all $k \geq 1$. This implies that $h \in \clh \ominus \clk$ if and only if $h$ is in the right side of \eqref{cnc}. Moreover, $\clh \ominus \clk$ is a $T^*$-invariant subspace of $\clh$. Consequently
\[
\clh_c: = \clh \ominus \clk =  \{h \in \clh: \sum_{|\al| = k } \|{T^*}^\al h \|^2 = \|h\|^2 \mbox{~for all ~} k \in \Z_+\},
\]
and $\displaystyle\sum_{i=1}^n W_i W_i^* = I_{\clh_c}$, where $W_i = P_{\clh_c}T_i|_{\clh_c}$ for all $i=1,\ldots,n$. Moreover
\[
W_i^*W_j^* = (T_i^*|_{\clh_c})(T_j^*|_{\clh_c}) = T_i^*T_j^*|_{\clh_c} = T_j^*T_i^*|_{\clh_c} = W_j^*W_i^*,
\]
for all $i, j=1, \ldots, n$. It follows that $W$ is a commuting spherical co-isometric tuple on $\clh_c$. Recall that an $n$-tuple $(X_1,\ldots, X_n)$ on $\cll$ is said to be a \textit{spherical co-isometry} if $\displaystyle\sum_{i=1}^n X_i X_i^* = I_{\cll}$.

Thus, we have proved:

\begin{Theorem}\label{structure-commutative}
Let $T=(T_1,\ldots,T_n)$ be a commuting row contraction on a Hilbert space $\clh$ with polynomial characteristic function of degree $m$. If $\clm = \overline{\text{span}} \{ T^{\al}D_{T^*}h : h \in \clh,
|\al| \geq m, \al \in \Z_+^n \}$, and
\[
\clh_{nil} = \overline{\text{span}}\{T^{\al}D_{T^*}h : h \in \clh, \al \in \Z_+^n\} \ominus \clm,
\]
and
\[
\clh_c =  \{h \in \clh: \sum_{|\al| = k } \|{T^*}^\al h \|^2 = \|h\|^2 \mbox{~for all ~} k \in \Z_+\},
\]
then $\clh = \clm \oplus \clh_{\text{nil}} \oplus\clh_c$ and $T_i$, $i=1,\ldots,n$ admits the following matrix decomposition
\begin{equation}\label{eq: T_i canonical}
T_i = \begin{bmatrix}
M_i &   *   &   *    \cr
    0       &  N_i  &   *    \cr
    0       &   0   &  W_i   \cr
\end{bmatrix},
\end{equation}
where $M$ on $\clm$ is a pure row contraction, $N$ on $ \clh_{\text{nil}}$ is a commuting nilpotent tuple of order less than or equal to $m$ and $W$ on $\clh_c$ is a commuting spherical co-isometry. Moreover,
\[
\displaystyle\sum_{i=1}^n M_i M_i^* = I_{\clm} - P_{\cln},
\]
where $\cln = \clm \ominus \Big(\displaystyle\sum_{i=1}^n {T}_i \clm\Big)$. If, in addition, $M$ is regular, then it is a Drury-Arveson shift.
\end{Theorem}

In the final section, we will study a non-trivial example of pure partial isometric commuting tuple whose characteristic function is not a polynomial. Therefore it is not unitarily equivalent to a Drury Arveson shift. Because of Corollary 3.10 of \cite{EL18}, a pure partial isometric commuting tuple is unitarily equivalent to a Drury Arveson shift if and only if it is a regular tuple. Thus,
the regularity assumption is an essential condition for the final conclusion in the above theorem.

For simplicity in what follows, we will refer to the representation \eqref{eq: T_i canonical} as simply the \textit{canonical representation} of $T$ with polynomial characteristic function of degree $m$. When the $n$-tuple $M$ of the canonical representation of $T$ is regular, we say that $T$ is \textit{regular}.

To avoid possible confusion, we remark in passing the following:
\begin{Remark}
If $m=0$, then the above construction yields that $\clh_{\text{nil}} = \{0\}$ and $ N_i = 0$ for all $i =1, \ldots, n$.
\end{Remark}

\section{Factorizations and noncommuting tuples}\label{sect-4}

We now turn to characteristic functions of noncommuting tuples introduced by Popescu \cite{Po89b}. Here, following \cite{FPS17}, we obtain an analytic structure of polynomial characteristic functions (up to unitary equivalence) of noncommuting row contractions.

The full \textit{Fock space} over $\mathbb C^n$, denoted by
$\Gamma$, is the Hilbert space
\[
\Gamma
:=\displaystyle\bigoplus_{m=0}^{\infty}(\mathbb{C}^n)^{\otimes^m}=
\mathbb{C}\oplus\mathbb{C}^n\oplus(\mathbb{C}^n)^{\otimes^2} \oplus
\cdots\oplus(\mathbb{C}^n)^{\otimes^m}\oplus \cdots.
\]
The \textit{vacuum vector} $1\oplus 0 \oplus \cdots \in \Gamma$ is
denoted by $e_\emptyset$. Let $\{e_1,\ldots, e_n\}$ be the standard
orthonormal basis of $\mathbb C^n$ and $\mathbb{F}_n^+$ be the
unital free semi-group with generators $1,\ldots, n$ and the
identity $\emptyset$. For $\alpha = \alpha_1\cdots\alpha_m \in
{\mathbb{F}_n^+}$ we denote the vector $e_{\alpha_1}\otimes \cdots
\otimes e_{\alpha_m}$ by $e_\alpha$. Then $\{e_\alpha: \alpha \in
{\mathbb{F}_n^+}\}$ forms an orthonormal basis of $\Gamma$. For each
$j = 1, \ldots, n$, the \textit{left creation operator} $L_j$ and the \textit{right creation operator} $R_j$ on $\Gamma$ are defined by
\[L_j f = e_j \otimes f, \quad \quad R_j f = f \otimes e_j \quad \quad (f \in
\Gamma),
\]
respectively. Moreover, $ R_j = U^* L_j U $ where $U$, defined by
\begin{equation}\label{eq-page 10}
U(e_{i_1}\otimes e_{i_2}\otimes\cdots \otimes e_{i_m}) =  e_{i_m}
\otimes \cdots \otimes e_{i_2}\otimes e_{i_1},
\end{equation}
is the \textit{flip operator} on $\Gamma$. The \textit{noncommutative disc algebra} ${\m A_n^\infty}$ is the norm closed algebra generated  by $\{I_{\Gamma}, L_1, \ldots,L_n \}$ and the \textit{noncommutative analytic Toeplitz
algebra} $\m F^\infty_n $ is the WOT-closure of ${\m A_n^\infty}$
(see Popescu \cite{Po95b}).

Let $\m E$ and $\m E_*$ be Hilbert spaces and $M \in \m B(\Gamma
\otimes \m E,  \Gamma \otimes \m E_*)$. Then $M$ is said to be
\textit {multi-analytic operator} if
\[
 M(L_i \otimes I_{\m E} ) = (L_i \otimes I_{\m E_*}) M  \quad \quad (i =1, \ldots,n).
\]
In this case, the bounded linear map $\theta \in \m B(\m E, \Gamma
\otimes \m E_*)$ defined by
\[\theta \eta = M(e_\emptyset \otimes \eta) \quad \quad (\eta \in
\m E),
\]
is said to be the \textit{symbol} of $M$ and we denote $M =
M_\theta$. Moreover, define $\theta_{\alpha} \in \m B(\m E, \m
E_*)$, $\alpha \in {\mathbb{F}_n^+} $ by
\[
\langle  \theta_{\alpha} \eta, \eta_* \rangle := \langle \theta
\eta, e_{\bar \alpha} \otimes \eta_* \rangle = \langle M
(e_\emptyset\otimes \eta),  e_{\bar \alpha} \otimes \eta_* \rangle,
\quad \quad (\eta \in \m E, \eta_* \in \m E_*)
\]
where $\bar \alpha$ is the reverse of $ \alpha$. The Fourier type representation for multi-analytic operators was considered first by Popescu  (see Popescu \cite{Po95a}), and from this representation, we have a unique formal Fourier expansion
\[
M ~\sim ~ \displaystyle\sum_{\alpha \in {\mathbb{F}_n^+}} R^\alpha
\otimes \theta_{\alpha},
\]
and
\[
M = \mbox{SOT}-\displaystyle\lim_{r \to 1^{-}}
\displaystyle\sum_{k=0}^{\infty}\displaystyle\sum_{|\alpha|=k}r^{|\alpha|}
R^\alpha \otimes \theta_{\alpha}
\]
where $ |\alpha|$ is the length of $ \alpha$. A multi-analytic operator $M_{\theta} \in \m B (\Gamma \otimes \m E, \Gamma \otimes \m E_*)$  is said to be \textit{purely contractive}
if $M_{\theta}$ is a contraction and
\[
\|P_{ e_\emptyset \otimes \m E_*} \theta \eta \| < \|\eta\| \quad
\quad (\eta \in \m E, \eta \neq 0).
\]
We say that $M_\theta$ \textit{coincides} with a multi-analytic operator
$M_{\theta'} \in \clb(\Gamma  \otimes \cle', \Gamma  \otimes \m
\cle_*')$ if there exist unitary operators $ W: \cle \to \cle'$ and
$ W_*:\cle_* \to \cle_*'$ such that
\[
(I_{\Gamma}\otimes W_*) M_{\theta} =  M_{\theta'}
(I_{\Gamma}\otimes W).
\]

Let $\m H$ be a Hilbert space and  $T = (T_1, \ldots, T_n)$ be a row
operator on $\m H$. \textsf{For simplicity of the notations, we will
denote by $\tilde T$ and $\tilde R$ the row operators $(I_\Gamma
\otimes T_1, \ldots, I_{\Gamma} \otimes T_n)$ and $(R_1 \otimes
I_{\m H}, \ldots, R_n \otimes I_{\m H})$ on $\Gamma \otimes \m H$,
respectively.}

The \textit{characteristic function} of a row contraction $T$ on $\m H$ is a purely contractive multi-analytic
operator $\Theta_{T} \in \m B(\Gamma \otimes \m D_{T}, \Gamma
\otimes \m D_{T^*})$ defined by
\[
\Theta_T \sim~ -I_\Gamma \otimes T + (I_\Gamma \otimes  D_{
T^*})(I_{\Gamma \otimes \m H}- \tilde R \tilde  T^*)^{-1} \tilde R
(I_\Gamma\otimes D_T).
\]
Hence
\[
\Theta_T  = \mbox{SOT}-\displaystyle\lim_{r \to 1}\Theta_T(r \tilde R),
\]
where for each $r \in [0, 1)$,
\[
\Theta_T(r \tilde R) := - \tilde T + D_{ {\tilde T}^*}(I_{\Gamma \otimes \m
H}-  r \tilde R\tilde T^*)^{-1} r \tilde R  D_{\tilde T}.
\]
Therefore
\begin{align}\label{3}
\Theta_T = \mbox{SOT}-\displaystyle\lim_{r \to 1}\Theta_T(r \tilde R) =
\mbox{SOT}-\displaystyle\lim_{r \to 1} \big[- \tilde T+ D_{ {\tilde T}^*}(I_{\Gamma \otimes \m H}-  r \tilde R\tilde T^*)^{-1} r \tilde
R  D_{\tilde T}\big].
\end{align}

Now we recall the classical result of Sz.-Nagy and Foias concerning $2 \times 2$ block contractions (see \cite{NF67}, and also \cite[Lemma 2.1, Chapter IV]{FF90}):

\begin{Theorem}\label{th:NFblock}
Let $\mathcal H_1$ and $\mathcal H_2$ be Hilbert spaces, $A =  (A_1, \ldots, A_n)\in \mathcal B (\mathcal H_1^n, \mathcal H_1)$, $B = (B_1, \ldots, B_n) \in \mathcal B(\mathcal H_2^n, \mathcal H_2)$ and $X = (X_1, \ldots, X_n) \in \mathcal B(\mathcal H_2^n,\mathcal H_1 )$ be row operators. Then the row operator
\[T =
\begin{bmatrix}A & X\\ 0 & B\end{bmatrix} \in \mathcal B(\mathcal H_1^n \oplus \mathcal H_2^n, \mathcal H_1 \oplus \mathcal H_2),
\]
is a row contraction if and only if $A$ and $B$ are row contractions and $X = D_{A^*} L D_{B}$ for some contraction $L \in  \mathcal B (\m D_{B}, \m D_{A^*})$.
\end{Theorem}

Next, we recall a result \cite[Theorem 2.2]{HMS17} concerning factorizations of characteristic functions of noncommutative tuples, which will be used in the proof of the main theorem of this section. Recall, given a contraction $L \in \clb(\clh, \clk)$, the \textit{Julia-Halmos matrix} corresponding to $L$ is defined by
\[
J_L =\begin{bmatrix} L^* & D_L
\\ D_{L^*} & - L
\end{bmatrix}.
\]

\begin{Theorem}\label{factor}
Let $\m H_1$ and $\m H_2$ be two Hilbert spaces. Suppose $A$ on $\m H_1$ and $B$ on $\m H_2$ are $n$-tuples of row contractions and $L \in \m B(\m D_B, \m D_{A^*})$ is a contraction, and let
\[
T =\begin{bmatrix}
A& D_{A^*}L D_B\\
0 & B
\end{bmatrix}:\m H_1^n \oplus \m H_2^n \to \m H_1 \oplus \m H_2.
\]
Then there exist unitaries $\tau \in \m B(\m D_T, \m D_A \oplus \m D_L)$ and $\tau_*	\in \m B(\m D_{T^*}, \m D_{B^*} \oplus \m D_{L^*})$ such that
\[
\Theta_T =(I_\Gamma \otimes \tau_*^{-1})
\begin{bmatrix}
\Theta_B & 0 \\
0 & I_{\Gamma \otimes  \m D_{ L^*}}
\end{bmatrix}
(I_\Gamma \otimes J_L)
\begin{bmatrix}
\Theta_A & 0\\ 0& I_{\Gamma \otimes \m D_{ L}}\end{bmatrix}
(I_\Gamma \otimes \tau),
\]
where $J_L \in \m B ( \m D_{A^*} \oplus \m D_L , \m D_B \oplus \m D_{L^*})$ is the Julia-Halmos matrix corresponding to $L$.
\end{Theorem}

We are now ready to prove the main factorization theorem of this section. The proof uses ideas similar to that used for Theorem 1.3 of \cite{FPS17}.

\begin{Theorem}\label{th:factorization}
Let $\clh$, $\m H_1$, $\m H_0$ and $\m H_{-1}$ be Hilbert spaces. Suppose $\m H = \m H_1 \oplus \m H_0 \oplus \m H_{-1}$ and assume that $T = (T_1, \ldots, T_n)$ is a row contraction on $\clh$ and
\[
T_i =\begin{bmatrix}
S_i & *&*\\
0 & N_i&*\\
0 & 0&C_i\\
\end{bmatrix},
\]
for all $i=1, \ldots, n$. Then $S$, $N$ and $C$ are $n$-tuples of row contractions on $\m H_1, \m H_0$ and $\m H_{-1}$, respectively, and there exist Hilbert spaces $\m E_1, \m E_2$ and $\m E$, and unitary operators
\[
\tau_1 \in \m B(\m D_{N^*} \oplus \m E, \m D_C \oplus\m E_1) \quad \mbox{and} \quad \tau_2 \in B(\m D_{S^*} \oplus \m E_2, \m D_N \oplus\m E),
\]
such that $\Theta_T$ coincides with
\[
\begin{bmatrix} \Theta_C & 0 \\
0 & I_{\Gamma \otimes  \m E_1}
\end{bmatrix}
(I_\Gamma \otimes \tau_1)
\begin{bmatrix}\Theta_N & 0\\
 0& I_{\Gamma \otimes \m E}
\end{bmatrix}
(I_\Gamma \otimes \tau_2)
\begin{bmatrix}
\Theta_S & 0\\
0& I_{\Gamma \otimes \m E_2}
\end{bmatrix}.
\]
\end{Theorem}
\begin{proof}
For each $i=1, \ldots, n$, set $T_i =   \begin{bmatrix}
A_i & Y_i\\
0 & C_i \end{bmatrix}$, where $A_i =  \begin{bmatrix}
S_i & X_i\\
0 & N_i \end{bmatrix}  = P_{\m H_1 \oplus \m H_0} T_i|_{\m H_1 \oplus \m H_0}$. Since $T$ is a row contraction, by Theorem \ref{th:NFblock}, $A$ and $C$ are row contractions and there exists a contraction $L_Y: \m D_C \to \m D_{A^*}$ such that $Y = D_{A^*} L_Y D_C$. On the other hand, since $A$ is a row contraction, by Theorem \ref{th:NFblock} again, it follows that $S$ and $N$ are row contractions and $X = D_{S^*} L_X D_N$ for some contraction $L_X: \m D_N  \to \m D_{S^*}$. Now, applying Theorem \ref{factor} to the row contraction $ T =  \begin{bmatrix}
A& D_{A^*} L_Y D_C\\
0 & C \end{bmatrix}$, we obtain
\[
\Theta_T =(I_\Gamma \otimes u_*^{-1})
\begin{bmatrix}
\Theta_C & 0 \\
0 & I_{\Gamma \otimes  \m D_{ L_Y^*}}
\end{bmatrix}
(I_\Gamma \otimes J_{L_Y})
\begin{bmatrix}
\Theta_A & 0\\ 0& I_{\Gamma \otimes \m D_{ L_Y}}\end{bmatrix}
(I_\Gamma \otimes u),
\]
for some unitary operators $u \in \m B (\m D_T, \m D_{A} \oplus \m D_{L_Y})$ and $u_* \in \m B (\m D_{T^*}, \m D_{C^*} \oplus \m D_{L_Y^*})$. Note that $J_{L_Y} \in \m B ( \m D_{A^*} \oplus \m D_{L_Y} , \m D_C \oplus \m D_{L_Y^*})$ is the Julia-Halmos matrix corresponding to $L_Y$. Again, applying Theorem \ref{factor} to the row contraction $A =  \begin{bmatrix}
S& D_{S^*} L_X D_N \\
0 & N \end{bmatrix}$, we obtain
\[
\Theta_A =(I_\Gamma \otimes \sigma_*^{-1})
\begin{bmatrix}
\Theta_N & 0 \\
0 & I_{\Gamma \otimes  \m D_{ L_X^*}}
\end{bmatrix}
(I_\Gamma \otimes J_{L_X})
\begin{bmatrix}
\Theta_S & 0\\ 0& I_{\Gamma \otimes \m D_{ L_X}}\end{bmatrix}
(I_\Gamma \otimes \sigma),
\]
for some unitary operators $\sigma \in \m B (\m D_A, \m D_{S} \oplus \m D_{L_X})$ and $\sigma_* \in \m B (\m D_{A^*}, \m D_{N^*} \oplus \m D_{L_X^*})$. Again note that $J_{L_X}\in \m B ( \m D_{S^*} \oplus \m D_{L_X} , \m D_N \oplus \m D_{L_X^*})$ is the Julia-Halmos matrix corresponding to $L_X$. For convenience, we denote
\[
 \Phi_S = \begin{bmatrix}
\Theta_S & 0\\ 0& I_{\Gamma \otimes \m D_{ L_X}}\end{bmatrix}, \Phi_N = 	\begin{bmatrix}
\Theta_N & 0 \\
0 & I_{\Gamma \otimes  \m D_{ L_X^*}}
\end{bmatrix},  ~\mbox{and}~ \Phi_C = \begin{bmatrix}
\Theta_C & 0 \\
0 & I_{\Gamma \otimes  \m D_{ L_Y^*}}
\end{bmatrix}.
\]
Therefore
\[
\begin{split}
\Theta_T
&=(I_\Gamma \otimes u_*^{-1})\Phi_C
(I_\Gamma \otimes J_{L_Y})
\begin{bmatrix}
(I_\Gamma \otimes \sigma_*^{-1})  \Phi_N (I_\Gamma \otimes J_{L_X}) \Phi_S  (I_\Gamma \otimes \sigma) & 0\\ 0& I_{\Gamma \otimes \m D_{ L_Y}}\end{bmatrix} (I_\Gamma \otimes u)
\\
& = (I_\Gamma \otimes u_*^{-1})\Phi_C
(I_\Gamma \otimes J_{L_Y})
\begin{bmatrix}
(I_\Gamma \otimes \sigma_*^{-1})  & 0\\ 0& I_{\Gamma \otimes \m D_{ L_Y}}
\end{bmatrix}  \begin{bmatrix}
 \Phi_N   & 0\\ 0& I_{\Gamma \otimes \m D_{ L_Y}}
\end{bmatrix}
\\
& \hspace*{.5cm} \times
\begin{bmatrix}
(I_\Gamma \otimes J_{L_X})  & 0\\ 0& I_{\Gamma \otimes \m D_{ L_Y}}
\end{bmatrix}
\begin{bmatrix}
\Phi_S   & 0\\ 0& I_{\Gamma \otimes \m D_{ L_Y}}
\end{bmatrix}
\begin{bmatrix}
I_\Gamma \otimes \sigma & 0\\ 0& I_{\Gamma \otimes \m D_{ L_Y}}\end{bmatrix}
(I_\Gamma \otimes u)
\\
& =  (I_\Gamma \otimes u_*^{-1})\Phi_C
(I_\Gamma \otimes \tau_1) \begin{bmatrix}
\Phi_N   & 0\\ 0& I_{\Gamma \otimes \m D_{ L_Y}}
\end{bmatrix}
(I_\Gamma \otimes \tau_2)
\begin{bmatrix}
\Phi_S   & 0\\ 0& I_{\Gamma \otimes \m D_{ L_Y}}
\end{bmatrix}
(I_\Gamma \otimes v).
\end{split}
\]
Here $\tau_1 \in \clb((\m D_{N^*} \oplus \m D_{L^*_X}) \oplus \m D_{L_Y}, \m D_C \oplus\m D_{L^*_Y})$, $\tau_2\in \clb((\m D_{S^*} \oplus \m D_{L_X}) \oplus \m D_{L_Y}, (\m D_N \oplus \m D_{L^*_X}) \oplus \m D_{L_Y})$ and $ \psi \in \clb(\m D_T, (\m D_S \oplus \m D_{L_X})\oplus \m D_Y)$ are unitary operators defined by
\[
I_{\Gamma}  \otimes \tau_1 = (I_\Gamma \otimes J_{L_Y})
\begin{bmatrix}
(I_\Gamma \otimes \sigma_*^{-1})  & 0\\ 0& I_{\Gamma \otimes \m D_{ L_Y}}
\end{bmatrix} \quad \mbox{and} \quad
I_{\Gamma}  \otimes \tau_2 = \begin{bmatrix}
(I_\Gamma \otimes J_{L_X})  & 0\\ 0& I_{\Gamma \otimes \m D_{ L_Y}}
\end{bmatrix},
\]
and
\[
I_\Gamma \otimes v = \begin{bmatrix} I_\Gamma \otimes \sigma & 0\\ 0& I_{\Gamma \otimes \m D_{ L_Y}}\end{bmatrix}
(I_\Gamma \otimes u).
\]
Hence
\[
\Theta_T
=(I_\Gamma \otimes u_*^{-1})
\begin{bmatrix}
\Theta_C & 0 \\
0 & I_{\Gamma \otimes  \m E_1}
\end{bmatrix}
(I_\Gamma \otimes \tau_1)\begin{bmatrix}
\Theta_N   & 0\\ 0& I_{\Gamma \otimes \m E}
\end{bmatrix}
(I_\Gamma \otimes \tau_2)
\begin{bmatrix}
\Theta_S   & 0\\ 0& I_{\Gamma \otimes \m E_2}
\end{bmatrix}(I_\Gamma \otimes v),
\]
where $\m E_1 = \m D_{L_Y^*}$, $\m E_2 = \m D_{L_X}\oplus \m D_{L_Y}$ and $\m E = \m D_{L_X^*} \oplus \m D_{L_Y}$. This completes the proof of the theorem.
\end{proof}

The following corollary is a noncommutative generalization of \cite[Theorem 2.2]{FPS17}.

\begin{Corollary}\label{cor1}
Assume the setting of Theorem \ref{th:factorization}. If $S$ is an isometry and $C$ is a spherical co-isometry, then there exist a Hilbert space $\m E$, a co-isometry $G_1 \in \clb (\Gamma \otimes (\m D_{N^*}\oplus \m E), \Gamma \otimes \m D_{T^*})$ and an isometry $G_2 \in \clb(\Gamma \otimes \m D_T, \Gamma \otimes (\m D_{N} \oplus \m E))$ such that
\[
\Theta_T = G_1 	\begin{bmatrix}
\Theta_N& 0 \\
0 & I_{\Gamma \otimes  \m D_{\m E}}
\end{bmatrix}G_2.
\]
\end{Corollary}
\begin{proof}
Since $\m D_S = \{0_{\m H_1^n}\}$ and $\m D_{C^*} = \{ 0_{\m H_{-1}}\}$, by assumption, it follows that the characteristic functions $\Theta_S : \Gamma \otimes \m D_S  \to \Gamma \otimes \m D_{S^*}$ and $\Theta_C: \Gamma \otimes \m D_C \to \Gamma \otimes \m D_{C^*}$ are identically zero, that is,
\[
0_S := \Theta_S \equiv 0 : \Gamma \otimes \{0_{\m H_1^n}\} \to \Gamma \otimes \m D_{S^*} \quad \mbox{ and} \quad 0_C := \Theta_C\equiv 0: \Gamma \otimes \m D_C \to \Gamma \otimes  \{0_{\m H_{-1}}\}.
\]
In this case, the unitary operators $u_*$, $\sigma$ and $v$ in the proof of Theorem \ref{th:factorization} become
\[
u_* \in \m B (\m D_{T^*},   \{ 0_{\m H_{-1}}\} \oplus \m D_{L_Y^*}) \quad \mbox{and} \quad \sigma \in \m B (\m D_A,  \{0_{\m H_1^n}\}   \oplus \m D_{L_X}),
\]
and $v \in \m B(\m D_T, ( \{0_{\m H_1^n}\}  \oplus \m D_{L_X})\oplus  \m D_Y)$, respectively. Then the representation of $\Theta_T$, as given in the final part of the proof of Theorem  \ref{th:factorization}, becomes
\[
\Theta_T
=G_1
\begin{bmatrix}
\Theta_N   & 0\\ 0& I_{\Gamma \otimes (\m D_{L^*_X}\oplus \m D_{ L_Y})}
\end{bmatrix}
G_2,
\]
where $\m E = \m D_{L^*_X}\oplus \m D_{ L_Y}$ and
\[
G_1 = (I_\Gamma \otimes u_*^{-1})
\begin{bmatrix}
0_C & 0 \\
0 & I_{\Gamma \otimes  \m D_{ L_Y^*}}
\end{bmatrix}
(I_\Gamma \otimes \tau_1)  \in \m B(\Gamma\otimes (\m D_{N^*}\oplus \m E), \Gamma \otimes \m D_{T^*}),
\]
and
\[
G_2 = 	(I_\Gamma \otimes \tau_2)
\begin{bmatrix}
0_S   & 0\\ 0& I_{\Gamma \otimes(\m D_{L_X}\oplus \m D_{ L_Y})}
\end{bmatrix}
(I_\Gamma \otimes v) \in   \m B( \Gamma\otimes \m D_T, \Gamma \otimes (\m D_{N}\oplus \m E)).
\]
Since $0_C 0_C^* = I_{\Gamma\otimes \{0_{\m H_{-1}}\}} $ and $0_S^*0_S = I_{\Gamma\otimes  \{ 0_{\m H_1^n}\}}$, it follows that $ G_1G_1^* = I_{\Gamma \otimes \m D_{T^*}}$ and $G_2^*G_2 = I_{\Gamma \otimes \m D_T}$. This completes the proof.
\end{proof}

In view of Popescu \cite[Theorem 1.1]{Po13}, the following corollary is now more definite:

\begin{Corollary}
Let $T$ be a be a row contraction on $\clh$ such that the characteristic function $\Theta_T$ is a noncommutative polynomial of degree $m$. Then there exist a Hilbert space $\m  E $, a nilpotent row contraction $N = (N_1, \ldots, N_n)$ of order $\leq m$, such that
\[
\Theta_T = G_1 	\begin{bmatrix}
\Theta_N& 0 \\
0 & I_{\Gamma \otimes  \m D_{\m E}}
\end{bmatrix}G_2,
\]
where $G_1$ and $G_2$ are co-isometry and isometry in $\clb (\Gamma \otimes (\m D_{N^*}\oplus \m E), \Gamma \otimes \m D_{T^*})$ and $\m B( \Gamma\otimes \m D_T, \Gamma \otimes (\m D_N \oplus \m E))$, respectively.
\end{Corollary}
\begin{proof}
Since $T$ is a row contraction with a noncommutative polynomial of degree $m$, by \cite[Theorem 1.1]{Po13}, there exist closed subspaces $\m H_{-1}$, $\m H_0$ and $\m H_1$ of $\m H$ such $\m H = \m H_{-1} \oplus \m H_0 \oplus \m H_1$ and the matrix representation of $T_i$ is given by
\[
T_i = \begin{bmatrix}
	S_i & *&*\\
	0 & N_i&*\\
	0 & 0&C_i\\
\end{bmatrix},
\]
for all $i=1, \ldots, n$, where $S$ is a row isometry on $\m H_1$, $N$ is a nilpotent row contraction of order $\leq m$ on $\m H_ 0$ and $C$ is a spherical co-isometry on $\m H_{-1}$. The remaining part of the proof now directly follows from Corollary \ref{cor1}.
\end{proof}

\section{Constrained Row Contractions and Factorizations}\label{sect-5}

In this section, we describe the factorization results obtained in the previous section in the setting of constrained row contractions. Constrained row contractions are related to the notion of noncommutative varieties, which was introduced by G. Popescu in \cite{Po06a}. The added complications and structures of noncommutative varieties are due mostly to the fact, as for example, that the Drury-Arveson space is a quotient space of the full Fock space (see \cite{Po06a, P10}).

First, we recall the basic features of constrained row contractions and noncommutative varieties and refer the reader to Popescu \cite{Po06a, P10} for further details.

Let $ \m P_J \subset \m F_n^\infty$ be a set of noncommutative polynomials, and let $J$ be the WOT-closed two sided ideal of $ \m F_n^\infty$ generated
by $\m P_J$. In what follows, we always assume that $J \neq \m
F_n^\infty$. Then
\[
\clm_J := \overline{\mbox{span}} \{\phi \otimes \psi: \phi \in J, \psi \in
\Gamma\} \quad \mbox{and} \quad \cln_J : = \Gamma\ominus \clm_J
\]
are proper joint $(L_1, \ldots, L_n)$ and $(L_1^*, \ldots, L_n^*)$
invariant subspaces of $\Gamma$, respectively. Define \textit{constrained
left creation operators} and \textit{constrained right creation operators} on
$\m N_J$ by
\[
V_j := P_{\m N_J }L_j|_{P_{\m N_J }}\quad \quad {\rm and} \quad
\quad W_j := P_{\m N_J }R_j|_{P_{\m N_J }} \quad \quad (j =1,
\ldots, n),
\]
respectively. Let $\m E$ and $\m E_*$ be Hilbert spaces, and let $M\in \z(\m N_J
\otimes \m E, \m N_J \otimes \m E_* )$. Then $M$ is said to be
\textit{constrained multi-analytic operator} if
\[
M(V_j \otimes I_{\m E})= (V_j \otimes I_{\m E_*})M \quad \quad \quad
(j =1,\ldots, n).
\]
A constrained multi-analytic operator $M\in \z(\m N_J \otimes \m E, \m N_J \otimes \m E_* )$ is said to be \textit{purely contractive} if $M$ is a contraction, $e_\emptyset \in \m N_J$ and
\[
\|P_{e_\emptyset \otimes \m E_*} M(e_\emptyset \otimes \eta)\| <
\|\eta \| \quad \quad (\eta \neq 0,  \eta \in \m E).
\]
Let $\m W(W_1,\ldots, W_n)$ denote the WOT-closed algebra generated by
$\{I, W_1,\ldots, W_n\}$ and $\m R_n^\infty = U^* \m F_n^\infty U$, where $U$ is the flipping operator (see \eqref{eq-page 10}). The following equality, due to Popescu \cite{Po06a}, is often useful:
\[
\m W(W_1,\ldots, W_n)~ \bar\otimes ~\z( \m E,\m E_* ) = P_{\m N_J
\otimes \m E_*}[ \m R_n^\infty~\bar \otimes~\z( \m E,\m E_* )
]|_{P_{\m N_J\otimes \m E}}.
\]
Recall also that a row contraction $T= (T_1, \ldots, T_n)$ on $\m H$ is said to be \textit{$J$-constrained row contraction}, or simply
\textit{constrained row contraction} if $J$ is clear from the
context, if
\[
p(T_1, \ldots, T_n) = 0 \quad \quad \quad (p \in \m P _J).
\]
The \textit{constrained characteristic function} $\Theta_{J, T}$ (see Popescu
\cite{Po06a}) of a constrained row contraction $T= (T_1, \ldots,
T_n)$ on a Hilbert space $\m H$ is defined by
\[
\Theta_{J, T}= P_{\m N_J \otimes \m D_{T^*}}\Theta_T|_{\m N_J
\otimes \m D_{T}}.
\]
Note that $\Theta_{J, T}$ is a pure constrained multi-analytic
operator $\Theta_{J, T}:\m N_J\otimes\m D_{T} \to \m N_J\otimes\m
D_{T^*}$. Moreover, since $ \m N_J \otimes \m D_{T^*}$ is a joint $(R_1^* \otimes I_{\m D_{T^*}}, \ldots,
R_n^* \otimes I_{\m D_{T^*}})$ invariant subspace and $W_i \otimes I_{\m D_{T^*}}= (P_{\m
N_J }R_i |_{P_{\m N_J }}) \otimes I_{\m D_{T^*}}$, $i =1, \ldots,
n$, it follows that (see \cite{Po06a})
\begin{gather}
\Theta_T^*(\m N_J \otimes \m D_{T^*}) \subset  \m N_J \otimes \m
D_{T} \quad \mbox{and} \quad \Theta_T(\m M_J \otimes \m D_{T})
\subset  \m M_J \otimes \m D_{T^*}.
\end{gather}

The starting point of our consideration of constrained row contractions is the following result \cite[Theorem 3.1]{HMS17}:

\begin{Theorem}\label{cons.factor.}
Let $A$ on $\m H_1$ and $B$ on $\m H_2$ be $n$-tuples of row contractions, and let $L \in \m B(\m
D_B, \m D_{A^*})$ be a contraction. If $T =
\begin{bmatrix}
A&  D_{A^*}L D_B\\
0 & B
\end{bmatrix}$ is a constrained row contraction on $\m H_1 \oplus \m
H_2$, then $A$ and $B$ are also constrained row contractions and there exist unitary operators $\sigma \in \m B(\m D_T, \m D_A \oplus \m D_L)$ and $\sigma_*
\in \m B(\m D_{T^*}, \m D_{B^*} \oplus \m D_{L^*})$ such that
\begin{align*}
\Theta_{J, T}= (I_{\m N}  \otimes \sigma_*^{-1})
\begin{bmatrix}
\Theta_{J, B} &  0 \\
0 &I_{\m N \otimes  \m D_{ L^*}}
\\
\end{bmatrix}
(I_{\m N}\otimes J_L)
\begin{bmatrix}
\Theta_{J,A}  & 0 \\
0 & I_{\m N \otimes \m D_{ L}} \\
\end{bmatrix}
(I_{\m N} \otimes \sigma)
\end{align*}
where $J_L \in \clb(\m D_{A^*} \oplus \m D_L, \m D_B \oplus \m D_{L^*})$ is the Julia-Halmos matrix corresponding to $L$.
\end{Theorem}

We are now ready to prove the factorization result for constrained row contractions. However, the proof is similar in spirit to that of Theorem \ref{th:factorization}, and thus, we only sketch it.

\begin{Theorem}\label{th:factorization for constrained rc}
Let $\clh$, $\m H_1$, $\m H_0$ and $\m H_{-1}$ be Hilbert spaces, $\m H = \m H_1 \oplus \m H_0 \oplus \m H_{-1}$. Suppose
\[
T_i =\begin{bmatrix}
S_i & *&*\\
0 & N_i&*\\
0 & 0&C_i\\
\end{bmatrix} \quad \quad (i=1, \ldots, n).
\]
If $T$ is a constrained row contraction, then $S$, $N$ and $C$ are also constrained row contractions on $ \m H_1 , \m H_0$ and $\m H_{-1} $, respectively, and there exist Hilbert spaces $\m E_1, \m E_2$ and $\m E$ and unitary operators $\tau_1 \in \m B(\m D_{N^*} \oplus \m E, \m D_C \oplus\m E_1)$ and $\tau_2 \in B(\m D_{S^*} \oplus \m E_2, \m D_N \oplus\m E)$ such that $\Theta_{J,T}$ coincides with
\[
\begin{bmatrix} \Theta_{J,C} & 0 \\
0 & I_{\cln \otimes  \m E_1}
\end{bmatrix}
(I_\cln \otimes \tau_1)
\begin{bmatrix}\Theta_{J,N} & 0\\
 0& I_{\cln \otimes \m E}
\end{bmatrix}
(I_\cln \otimes \tau_2)
\begin{bmatrix}
\Theta_{J,S} & 0\\
0& I_{\cln \otimes \m E_2}
\end{bmatrix}.
\]
\end{Theorem}
\begin{proof}
We use the same notations as in the proof of Theorem \ref{th:factorization}: $ T_i =   \begin{bmatrix}
A_i & Y_i\\
0 & C_i \end{bmatrix}$, where $A_i =  \begin{bmatrix}
S_i & X_i\\
0 & N_i \end{bmatrix}  = P_{\m H_1 \oplus \m H_0} T_i|_{\m H_1 \oplus \m H_0}$ for all $i=1, \ldots, n$. By Theorem \ref{th:NFblock} and first part of Theorem \ref{cons.factor.}, we already know that $A$ and $C$ are constrained row contractions and $Y = D_{A^*} L_Y D_C$ for some contraction $L_Y: \m D_C \to \m D_{A^*}$. Repeating the argument to the constrained row contraction $A$, we obtain that $S$ and
$N$ are also constrained row contractions and $X=  D_{S^*} L_X D_N$ for some contraction $L_X: \m D_N  \to \m D_{S^*}$. Then, applying Theorem \ref{cons.factor.} to the constrained row contractions $ T =  \begin{bmatrix}
A& D_{A^*} L_Y D_C\\
0 & C \end{bmatrix}$ and $A =  \begin{bmatrix}
S& D_{S^*} L_X D_N \\
0 & N \end{bmatrix} $, we find
\[
\Theta_{J,T} =(I_\cln \otimes u_*^{-1})
\begin{bmatrix}
\Theta_{J,C} & 0 \\
0 & I_{\cln \otimes  \m D_{ L_Y^*}}
\end{bmatrix}
(I_\cln \otimes J_{L_Y})
\begin{bmatrix}
\Theta_{J,A} & 0\\ 0& I_{\cln \otimes \m D_{ L_Y}}\end{bmatrix}
(I_\cln \otimes u),
\]
for some unitary operators $u \in \m B (\m D_T, \m D_{A} \oplus \m D_{L_Y})$ and $u_* \in \m B (\m D_{T^*}, \m D_{C^*} \oplus \m D_{L_Y^*})$, and
\[
\Theta_{J,A} =(I_\cln \otimes \sigma_*^{-1})
\begin{bmatrix}
\Theta_{J,N} & 0 \\
0 & I_{\cln \otimes  \m D_{ L_X^*}}
\end{bmatrix}
(I_\cln \otimes J_{L_X})
\begin{bmatrix}
\Theta_{J,S} & 0\\ 0& I_{\cln \otimes \m D_{ L_X}}\end{bmatrix}
(I_\cln \otimes \sigma),
\]
for some unitary operators $\sigma \in \m B (\m D_A, \m D_{S} \oplus \m D_{L_X})$ and $\sigma_* \in \m B (\m D_{A^*}, \m D_{N^*} \oplus \m D_{L_X^*})$. Finally, by the same reasoning as in the proof of Theorem \ref{th:factorization}, it follows that
\[
\Theta_{J,T}
=(I_\cln \otimes u_*^{-1})
\begin{bmatrix}
\Theta_{J,C} & 0 \\
0 & I_{\cln \otimes  \m E_1}
\end{bmatrix}
(I_\cln \otimes \tau_1)\begin{bmatrix}
\Theta_{J,N}   & 0\\ 0& I_{\cln \otimes \m E}
\end{bmatrix}
(I_\cln \otimes \tau_2)
\begin{bmatrix}
\Theta_{J,S}   & 0\\ 0& I_{\cln \otimes \m E_2}
\end{bmatrix}(I_\cln \otimes v),
\]
where $\m E_1 = \m D_{L_Y^*}$, $\m E_2 = \m D_{L_X}\oplus \m D_{L_Y}$ and $\m E = \m D_{L_X^*} \oplus \m D_{L_Y}$, and $\tau_1: (\m D_{N^*} \oplus \m D_{L^*_X}) \oplus \m D_{L_Y} \raro \m D_C \oplus\m D_{L^*_Y}$, $\tau_2: (\m D_{S^*} \oplus \m D_{L_X}) \oplus \m D_{L_Y} \raro (\m D_N \oplus \m D_{L^*_X}) \oplus \m D_{L_Y}$ and
$ v : \m D_T \raro (\m D_S \oplus \m D_{L_X})\oplus \m D_{L_Y}$ are unitary operators defined by
\[
I_{\cln}  \otimes \tau_1 = (I_\cln \otimes J_{L_Y})
\begin{bmatrix}
(I_\cln \otimes \sigma_*^{-1})  & 0\\ 0& I_{\cln \otimes \m D_{ L_Y}}
\end{bmatrix} \quad \mbox{and} \quad
I_{\cln}  \otimes \tau_2 = \begin{bmatrix}
(I_\cln \otimes J_{L_X})  & 0\\ 0& I_{\cln \otimes \m D_{ L_Y}}
\end{bmatrix},
\]
and
\[
I_\cln \otimes v = \begin{bmatrix} I_\cln \otimes \sigma & 0\\ 0& I_{\cln \otimes \m D_{ L_Y}}\end{bmatrix}
(I_\cln \otimes u).
\] This completes the proof of the theorem.
\end{proof}

The particular case of constrained row contractions where the noncommutative variety is given by
\[
\clp_{J_c}= \{L_iL_j - L_jL_i : i,j=1,\ldots,n \},
\]
gives rise to commuting row contractions on Hilbert spaces. In this case, $\cln_{J_c}$ becomes the symmetric Fock space $\Gamma_s$ and the $n$-tuple $V$ on $\Gamma_s$, where $V_j= P_{\Gamma_s}L_j|_{\Gamma_s}$, $j=1,\ldots,n$, becomes the left creation operators on $\Gamma_s$ (see \cite{B", Po06a},). More specifically, $V$ on $\cln_{J_c}$ and $(M_{z_1}, \ldots, M_{z_n})$ on $H^2_n$ are unitarily equivalent, where $H^2_n$ is the Drury-Arveson space and $M_{z_i}$ is the multiplication operator by the coordinate function $z_i$ on $H^2_n$, $i=1, \ldots, n$. Under this identification, $P_{\Gamma_s}\clf_n^{\infty}|_{\Gamma_s}$ corresponds to $\clm(H^2_n)$, the multiplier algebra of $H^2_n$ (see also Section \ref{sect-2}).

From this point of view, if $T$ on $\clh$ is a constrained row contraction corresponding to $\clp_{J_c}$, then $T_iT_j = T_j T_i$ for all $i,j=1, \ldots, n$, and one can identify the constrained characteristic function $ \Theta_{J_c,T} = P_{\cln_{J_c} \otimes \cld_{T^*}} \Theta_T|_{\cln_{J_c} \otimes \cld_T}$ with the $\clb(\cld_T, \cld_{T^*})$-valued multiplier $\theta_T: \B^n \to \clb(\cld_T, \cld_{T^*})$ in $\clm(\cld_T, \cld_{T^*})$ \cite{BES05, B", Po06a}, the characteristic function of the commuting tuple $T$ (see \eqref{eq-ch fn commutative}). In the remaining part of this paper, the identification of $\Theta_{J_c,T}$ and $\theta_T$ will be used interchangeably.

The first half of the following theorem is essentially a particular (the commutative) case of Theorem \ref{th:factorization for constrained rc}. The partially isometric property of $V_2$ in the remaining part is a special feature of $n$-tuples, $n>1$, of commuting row contractions.

\begin{Theorem}\label{converse}
Let $\m H_1$, $\m H_0$ and $\m H_{-1}$ be Hilbert spaces,  and let  $ T = (T_1, \ldots, T_n) $ be an $ n $-tuple of commuting row contraction on $\m H_1 \oplus  \m H_0 \oplus \m H_{-1}$ such that each $ T_i$ has the following matrix representation
\[
T_i =\begin{bmatrix}
S_i & *&*\\
0 & N_i&*\\
0 & 0&C_i\\
\end{bmatrix} \quad \quad (i =1, \ldots, n).
\]
  with respect to $\m H_1 \oplus  \m H_0 \oplus \m H_{-1}$. Then $S$ on $\clh_1$, $N$ on $\m H_0$ and $C$ on $\m H_{-1}$ are commuting row contractions, and there exist Hilbert spaces $\m E_1, \m E_2$ and $\m E$, and unitary operators $U_1 \in \m B(\m D_{N^*} \oplus \m E, \m D_C \oplus\m E_1)$ and $U_2 \in B(\m D_{S^*} \oplus \m E_2, \m D_N \oplus\m E)$ such that $\theta_T$ coincides with
\[
\begin{bmatrix}
\theta_C & 0 \\
0 & I_{H^2_n \otimes  \m E_1}
\end{bmatrix}
(I_{H^2_n} \otimes U_1)
\begin{bmatrix}\theta_N & 0\\
0& I_{H^2_n \otimes \m E}
\end{bmatrix}
(I_{H^2_n} \otimes U_2)
\begin{bmatrix}
\theta_S & 0\\
0& I_{H^2_n \otimes \m E_2}
\end{bmatrix}.
\]
In addition, if $S$ and $C$ are  Drury-Arveson shift and spherical co-isometry, respectively, then there exist a Hilbert space $ \m E $, a co-isometry $ G_1 \in \m (H^2_n \otimes (\m D_{N^*}\oplus \m E), H^2_n \otimes \m D_{T^*} )$  and a partial isometry $ G_2 \in \m (H^2_n \otimes \m D_T, H^2_n \otimes (\m D_{N} \oplus \m E))$ such that
\[
\theta_T = G_1 	\begin{bmatrix}
\theta_N& 0 \\
0 & I_{\Gamma \otimes  {\m E}}
\end{bmatrix}G_2.
\]
		
\end{Theorem}
\begin{proof} We only need to prove the second half. Since $C$ is a spherical co-isometry, $D_{C^*} = 0$, and hence $\m D_{C^*} =\{0_{\m H_{-1}}\} $. On the other hand, since $S$ is the Drury-Arveson shift,
$\theta_S$ is identically zero \cite[Proposition 2.6]{Po07}. Therefore, the unitary operators $ u_* ,\sigma$ and $v$ and the characteristic function $\Theta_{J,T}$, in terms of $\theta_T$, in the proof of Theorem \ref{th:factorization for constrained rc} becomes
\[
u_* \in \m B (\m D_{T^*},   \{ 0_{\m H_{-1}}\} \oplus \m D_{L_Y^*}), \; \sigma \in \m B (\m D_A, \m D_{S} \oplus \m D_{L_X}), \; \mbox{and}\; v \in \m B(\m D_T, (\m D_S \oplus \m D_{L_X})\oplus  \m D_{L_Y}),
\]
and
\begin{align*}
\theta_T
&=G_1
\begin{bmatrix}
\theta_N   & 0\\ 0& I_{H^2_n\otimes (\m D_{L^*_X}\oplus \m D_{ L_Y})}
\end{bmatrix}
G_2,
\end{align*}
respectively,
where $\m E = \m D_{L^*_X}\oplus \m D_{ L_Y}$,
\[
G_1 = (I_{H^2_n} \otimes u_*^{-1})
\begin{bmatrix}
0_C & 0 \\
0 & I_{H^2_n} \otimes  \m D_{ L_Y^*}
\end{bmatrix}
(I_{H^2_n}\otimes U_1)  \in \m B({H^2_n}\otimes (\m D_{N^*}\oplus \m E), {H^2_n}\otimes \m D_{T^*} ),
\]
\[
G_2 = 	(I_{H^2_n}\otimes U_2)
\begin{bmatrix}
0_S   & 0\\ 0& I_{{H^2_n}\otimes(\m D_{L_X}\oplus \m D_{ L_Y})}
\end{bmatrix}
(I_{H^2_n}\otimes v) \in   \m B( {H^2_n}\otimes \m D_T,{H^2_n} \otimes (\m D_{N}\oplus \m E)).
\]
Since $0_C 0_C^* = I_{{H^2_n}\otimes \{0_{\m H_{-1}}\} }$ and $0_S \equiv 0$, it is clear that $G_1$ is an co-isometry and $G_2$ is a partial isometry.
\end{proof}

If $N$ on $\m H_0$ is nilpotent of order $m$, then the above result clearly yields that the characteristic function $\theta_T$ is a polynomial of degree $ \leq m$. Moreover, in view of Theorem \ref{structure-commutative}, we have the following:

\begin{Theorem}\label{Theorem-Canonical}
Let $ T = (T_1, \ldots, T_n) $ be an $n$-tuple of commuting row contraction on a Hilbert space $ \m H $ such that $\theta_T$ is a polynomial of degree $m$, and let
\[
\clm = \overline{\text{Span}} \{ T^{\al}D_{T^*}h : h \in \clh, |\al| \geq m, \al \in \Z_+^n \}.
\]
If $T$ is regular, then there exist a Hilbert space $\m  E $, a co-isometry $ G_1 \in \m B ({H^2_n} \otimes (\m D_{N^*}\oplus \m E), {H^2_n} \otimes \m D_{T^*}) $ and a partial isometry $ G_2 \in \m B( {H^2_n}\otimes \m D_T, {H^2_n} \otimes (\m D_N \oplus \m E))$ such that
\[
\theta_T = G_1 	\begin{bmatrix}
\theta_N& 0 \\
0 & I_{{H^2_n} \otimes  \m D_{\m E}}
\end{bmatrix}G_2.
\]
\end{Theorem}

Note that if $n=1$, then the partial isometry $G_2$ becomes an isometry (see \cite[Theorem 2.2]{FPS17}).

\newsection{Uniqueness of the canonical representations}\label{sect-6}

Recall that the canonical representation of a commuting row contraction $T$ with polynomial characteristic function of degree $m$ is the upper triangular representation of $T_i$ on $\clh = \clm \oplus \clh_{nil} \oplus \clh_c$ as in Theorem \ref{structure-commutative}. In this section, we analyze the structure of canonical representations of commuting row contractions with polynomial characteristic functions.

We first prove that $\clm$ and $\clh_c$ of the canonical representation are optimal in an appropriate sense (see \cite[Proposition 2.1]{Po13} for $n$-tuples of noncommutative row contractions).

\begin{Theorem}\label{Theorem-Uniqueness}
Let $T$ be an $n$-tuple of commuting row contraction on a Hilbert space $ \m H $ such that $\theta_T$ is a polynomial of degree $m$ and also $T$ is regular. Suppose $\clh_1$, $\clh_0$ and $\clh_{-1}$ are Hilbert spaces, and let $\clh = \clh_1 \oplus \clh_0 \oplus \clh_{-1}$. If the matrix representation of $T_i$ with respect to $\clh = \clh_1 \oplus \clh_0 \oplus \clh_{-1}$ is given by
\[
T_i = \begin{bmatrix}
M_i^{'} & * &* \cr
0 & N_i^{'} & * \cr
0 & 0 & W_i^{'}   \cr
\end{bmatrix} \quad \quad (i=1, \ldots, n),
\]
where $M'$ on $\clh_1$ is a Drury-Arveson shift, $N'$ on $\clh_0$ is nilpotent of order $m$ and $W^{'}$ on $\clh_{-1}$ is a spherical co-isometry, then $ \m M \subseteq \m H_1$ and $\m H_c \supseteq \m H_{-1}$.
\end{Theorem}
\begin{proof} Since $W'$ is a spherical co-isometry, with respect to $\clh = \clh_1 \oplus \clh_0 \oplus \clh_{-1}$, we have $D_{T^*}^2  = \begin{bmatrix}
*&   *   &   *    \cr
* &  * &   *    \cr
*  &  *&  0 \cr
\end{bmatrix}$. If $D_{T^*} = \begin{bmatrix} * & * & A_{31} \\ * & * & A_{32}\\ A_{31} & A_{32} & A_{33}\end{bmatrix}$, then $D_{T^*}^2  = \begin{bmatrix}
*&   *   &   *    \cr
* &  * &   *    \cr
*  &  *&  A_{31}A_{31}^*+ A_{32} A_{32}^* + A_{33} A_{33}^*\cr
\end{bmatrix}$, and hence $A_{31}A_{31}^*+ A_{32} A_{32}^* + A_{33} A_{33}^* = 0$. It follows that $ A_{31} = A_{32} = A_{33} = 0 $. Therefore $D_{T^*}  = \begin{bmatrix}
*&   *   &   0    \cr
* &  * &   0    \cr
0 & 0 &  0 \cr
\end{bmatrix}$, and hence on $\m H= \m H_1 \oplus   \m H_0 \oplus  \m H_{-1} $, we have
\[
T^\alpha D_{T^*}  =  \begin{bmatrix}
M^{'\alpha} &   *   &   *  \cr
0&  0 &   *   \cr
0 & 0 &  W^{'\alpha}  \cr
\end{bmatrix}
\begin{bmatrix}
*&   *   &   0    \cr
* &  * &   0    \cr
0 & 0 &  0 \cr
\end{bmatrix}
= 		\begin{bmatrix}
*&   *   &   0    \cr
0& 0 &   0    \cr
0 & 0 &  0 \cr
\end{bmatrix},
\]
and $T^\alpha D_{T^*} \clh \subseteq \m H_1$ for all $\alpha \in \Z_+^n$ with $ |\alpha|\geq m$. Then, $\m M \subseteq \m H_1 $, where, on the other hand, $\m H_{-1} \subseteq  \m H_c$ as $\m H_c$ is the maximal closed joint $T^*$ invariant subspace of $\clh$ such that 	$\begin{bmatrix}T_1^*|_{\m H_c} \cr \vdots \cr T_n^*|_{\m H_c}\cr\end{bmatrix}$ is an isometry.
\end{proof}

Now we prove that the diagonal entries of the canonical representation of $T$, as in Theorem \ref{structure-commutative}, is a complete unitary invariant. The noncommutative version of this is due to Popescu \cite[Proposition 2.1]{Po13}.

\begin{Proposition}
Let $T$ on $\clh$ and $T'$ on $\clh'$ be $n$-tuples of commuting row contractions with polynomial characteristic functions of degree $m$. Assume that
\[
T_i= \begin{bmatrix}
M_i    &     *     &     *    \cr
 0         &     N_i   &     *    \cr
 0         &     0     &     W_i  \cr
\end{bmatrix} ~~\text{and}~~~~~  T_i^{'} = \begin{bmatrix}
M_i^{'}&      *        &     *    \cr
 0           &     N_i^{'}   &     *    \cr
 0           &      0        &   W_i^{'}\cr
\end{bmatrix} \quad \quad (i=1,\ldots,n),
\]
are the canonical representations of $T$ and $T'$ on
$\clh= \clm \oplus \clh_{\text{nil}} \oplus \clh_c$ and $\clh^{'} = \clm^{'} \oplus \clh_{\text{nil}}^{'} \oplus \clh_c^{'}$ respectively. If $U: \clh \to \clh^{'}$ is a unitary operator such that $UT_i = T_i'U$ for all $i=1, \ldots, n$, then $U\clm= \clm^{'}$, $U\clh_{\text{nil}}=\clh_{\text{nil}}^{'}$ and $U\clh_c=\clh_c^{'}$, and $(U|_{\clm})M_i = M_i^{'}(U|_{\clm})$, $(U|_{\clh_{\text{nil}}})N_i = N_i^{'}(U|_{\clh_{\text{nil}}})$ and $(U|_{\clh_c})W_i = W_i^{'}(U|_{\clh_c})$ for all $i=1,\ldots,n$.
\end{Proposition}
\begin{proof} Clearly, $U T_j  = T_j{'} U$ and $ U T_j^*  = T_j^{'*}U$ for $ j =1, \ldots, n $, implies that $U T^{\alpha} D_{T^*} = T^{'\alpha} D_{T^{'*}}U$, $\alpha \in \Z_+^n $, and, on the other hand, we have by definition $U\clm=\clm^{'}$. Moreover, since
\[
\|T^{'*\alpha}(Uh)\|^2 = \|U T^{*\alpha} h \|^2 = \|T^{*\alpha} h \|^2,
\]
for all $\alpha \in \Z_+^n $ and $ h \in \m H_c $, it follows that $ U\m H_c = \m H_c ^{'}$, and hence $U\clh_{\text{nil}}= \clh_{\text{nil}}^{'}$. The remaining part now follows from the representation $U= \begin{bmatrix}
U|_{\clm}    &         0                &         0             \cr
0         & U|_{\clh_{\text{nil}}}   &         0             \cr
0         &         0                &    U|_{\clh_c}        \cr
\end{bmatrix}$.
\end{proof}

For convenience, and following Popescu \cite{Po13}, we introduce the following notation. Denote $\Z_+ \cup \{\infty\}$ by $\N_{\infty}$ and denote by $\clc_n$ the set of all $n$-tuples of commuting row contractions on Hilbert spaces. We define $\vp: \clc_n \raro \N_{\infty} \times \N_{\infty} \times \N_{\infty}$ as follows: Let $T$ be an $n$-tuple of commuting row contraction on $\clh$. Define
\[
\vp(T) = (p,m,q),
\]
where $m:=\deg \theta_T$,
$q:=\dim\{h \in \clh : \sum_{|\al|=k} \|{T^*}^\al h \|^2 = \|h\|^2, \;\; \mbox{for all}~ k \, \in \Z_+ \}$ and
\[ p :=
\begin{cases}
\dim (\cld_m \ominus \cld_{m+1})& \quad  \text{if } m\in \Z_+ \\
\dim \cld_{T^*}                 & \quad  \text{if } m= \infty,
\end{cases}
\]
and $\cld_m:= \overline{\text{span}}\{ T^{\al}D_{T^*}h : h\in \clh, |\al| \geq m \}$.

Clearly, if a pair $T$ and $T^{'}$ in $\clc_n$ are unitarily equivalent, then $\vp(T)= \vp(T^{'})$. For $T \in \clc_n$ such that $\vp(T) \in \N_{\infty} \times \{0\} \times \{0\}$, we have the following:

\begin{Theorem}
Let $T, T' \in \clc_n$.
\begin{enumerate}
\item [(i)] $T$ is a Drury-Arveson shift if and only if $T$ is regular and $\vp(T) \in \N_{\infty} \times \{0\} \times \{0\}$.
\item [(ii)] If $T$ and $T^{'}$ are regular and $\vp(T)=\vp(T^{'}) = (p,0,0)$ for some $p\in \N_{\infty}$, then $T$ and $T^{'}$ are unitary equivalent and $\text{rank} D_{T^*} = \text{rank} D_{T^{'*}} = p$.
\end{enumerate}
\end{Theorem}
\begin{proof} (i) To prove the necessary part, without loss of generality, assume that $T$ is $M_z$ on $H^2_n(\cle)$ for some Hilbert space $\cle$. Observe that since $\theta_{M_z} \equiv 0$, we have $m=0$. Also note that $ D_{T^*}  = P_{\C} \otimes I_{\cle}$, which implies $\m  D_{T^*}  = \C \otimes \cle$. Since
\[
\overline{\text{span}} \{ T^{\al}D_{T^*}h : h \in \clh, \al \in \Z_+^n \} \ominus \overline{\text{span}} \{ T^{\al}D_{T^*}h : h \in \clh, |\al| \geq 1 \} = \C \otimes \cle,
\]
it follows that $p= \dim \cle = \text{rank}D_{T^*} \in \N_{\infty}$. Since $T$ is pure, $\clh_{c}=\{0\}$, and so $q = \text{dim}\clh_{c} = 0$. Thus $\vp(T) \in \N_{\infty} \times \{0\} \times \{0\}$. For the converse, assume that $T$ is regular and $\vp(T)  \in \N_{\infty} \times \{0\} \times \{0\}$. Therefore, since $ m= q=0 $, Theorem \ref{structure-commutative} implies that $\m H_{\text{nil}} =\{0\}, \m H_c = \{0\}$, that is, $T$ is a Drury-Arveson shift.

(ii)  This follows from part (i) and the fact that the multiplicity is a complete set of unitary invariant of Drury-Arveson shifts.
\end{proof}

The noncommutative version of the above result is due to Popescu \cite[Theorem 2.2]{Po13}. Now we turn to pure row contractions in $\clc_n$. The proof is completely analogous to the proof of \cite[Theorem 2.4 (i)]{Po13}.

\begin{Proposition}
Let $T$ be an $n$-tuple of commuting row contraction with polynomial characteristic function. Then $T$ is pure if and only if $\vp(T) \in \N_{\infty} \times \Z_+  \times \{0\}$.
\end{Proposition}
\begin{proof} Assume that $T$ is pure. Consider the canonical representation of $T$ on $\clm \oplus \clh_{\text{nil}} \oplus \clh_c$ as in Theorem \ref{structure-commutative}. For each $h \in \clh_c$, it follows that $T^{*\al} h = W^{*\al}h$ and hence
\[
\|h\|^2 = \sum_{|\al|=k} \|W^{*\al}h \|^2 = \sum_{|\al|=k} \|T^{*\al} h\|^2,
\]
for all $k \in \N$. Since $T$ is pure, this implies that $\clh_c = \{0\}$, that is, $\vp(T) \in \N_{\infty} \times \Z_+ \times \{0\}$.
Conversely, if $\varphi(T) \in \N_{\infty} \times \Z_+ \times \{0\}$, then $\clh_c = \{0\}$. The canonical representation of $T$ as in Theorem \ref{structure-commutative} then becomes $T_i = \begin{bmatrix}
M_i  &   *   \cr
0        &  N_i  \cr
\end{bmatrix}$ on $\clh = \clm \oplus \clh_{\text{nil}}$. Suppose $m$ is the order of the nilpotent operator $N$. Then for each $\alpha \in \Z_+^n$, $|\al| = m$, there exists $X_{\al} \in \clb(\clh_{\text{nil}},\clm) $ such that $ T^{\al} = \begin{bmatrix}
M^{\al}  &    X_{\al}   \cr
0          &    0         \cr
\end{bmatrix}$. By a computation similar to that in \cite[Theorem 2.4 (i)]{Po13}, we obtain that $T$ is pure.
\end{proof}

Along similar lines, most of Popescu's results in \cite[Section 2]{Po13} hold in a similar way for $n$-tuples of commuting contractions. We only point one which needs an additional assumption.

\begin{Theorem}
Let $T$ be an $n$-tuple of commuting contractions on a Hilbert space with polynomial characteristic function. If $T$ is regular, then the following are equivalent:
\begin{enumerate}
\item $\theta_T$ is constant.
\item $\varphi(T) \in \N_{\infty} \times \{0\} \times \N_{\infty} $.
\item The canonical decomposition of $T$ is given by: $T_i= \begin{bmatrix}
M_i   &    *    \cr
 0        &    W_i  \cr
\end{bmatrix}$ on  $\clh = \m M \oplus \clh_c$, $i=1,\ldots,n$, where $(M_1,\ldots, M_n)$ is a Drury-Arveson shift on $\m M$ and $(W_1,\ldots,W_n)$ is a spherical co-isometry on $\clh_c$.
\end{enumerate}
\end{Theorem}
\begin{proof} The proof follows from the definition of the map $\varphi$ and the canonical representation of the row contraction $T$ with polynomial characteristic function.
\end{proof}

\newsection{An example}\label{sect-7}

In this section, we provide an example of a commuting tuple, which is a partial isometry with wandering subspace property, but whose characteristic function is not a polynomial. Therefore, the tuple is not unitarily equivalent to a Drury Arveson shift.
This justifies the presence of the regularity assumption in the
Theorem \ref{structure-commutative}.

We consider a subspace of $H^2_2$ which is invariant under the Drury-Arveson shift. That is,
\[\mathcal{M}:= \bigoplus_{n\geq 2} \mathbb{H}_n \subseteq H^2_2, \] where $\mathbb{H}_n$ denotes the class of homogeneous polynomials of degree $n$ and a commuting pair of bounded linear operators $V=(V_1,V_2)$ on $\mathcal{M}$, defined by
\[V_i:= M_{z_i}|_{\mathcal{M}}\quad \quad \text{for}\,\,\, i=1,2.\]

By using the definition of the adjoint of the Drury-Arveson shift, 
it is trivial to observe that for each $i=1$ and $2$, 
\[ V_iV_i^* \bm z^{\alpha} = \begin{cases} \frac{\alpha_i}{|\alpha|} \bm z^{\alpha} & \mbox{if}~ |\alpha| \geq 3
\\
\displaystyle 0 &  \mbox{otherwise}. \end{cases}\]
From the definition of $V_i$'s, one can easily derive $D_{V^*}^2 = I - \sum_{i=1}^2 V_iV_i^* = P_{\mathbb{H}_2}$, where $P_{\mathbb{H}_2}$ is an orthogonal projection onto the subspace $\mathbb{H}_2$. Hence, $V$ is a row contraction on $\mathcal{M}$, and we recall the expression of the characteristic function of $V$, given in \ref{eq-ch fn commutative}, and the Taylor series expansion, that is,
\begin{align*}
\Theta_V(z) &= [-V + D_{V^*}(I - ZV^*)^{-1}ZD_V]|_{\mathcal{D}_V} \\
            &= (-V + \sum_{|\alpha| \geq 1} \Theta_{V,\alpha} \bm z^{\alpha})|_{\mathcal{D}_V},
\end{align*} where for each $\alpha$ with $|\alpha| \geq 1$ the coefficients
$\Theta_{V,\alpha} = \sum_{i=1}^2 \gamma_{\alpha - e_i} D_{V^*}V^{*(\alpha -e_i)}P_iD_V$. Also due to the fact that, $\mbox{Im}V \subseteq \bigoplus_{n \geq 3} \mathbb{H}_n$, we have $VD_V = D_{V^*}V = P_{\mathbb{H}_2}V = 0$.

On the other hand, from the definition of the defect operator $D_V^2: \mathcal{M} \oplus \mathcal{M} \to \mathcal{M} \oplus \mathcal{M}$, the action on the elements $( z_1^{\alpha_1}, z_1^{\alpha_1} )^{tr}$ with $\alpha_1 \geq 2$ is the following
\[ D_V^2 \begin{bmatrix}
z_1^{\alpha_1} \cr z_1^{\alpha_1}
\end{bmatrix} = \begin{bmatrix}
                 I-V_1^*V_1  &   -V_1^*V_2 \cr
                 -V_2^*V_1   &   I-V_2^*V_2
                \end{bmatrix} \begin{bmatrix}
                              z_1^{\alpha_1} \cr z_1^{\alpha_1}
                              \end{bmatrix}
                              = \begin{bmatrix}
                           \frac{-\alpha_1}{\alpha_1+1} z_1^{\alpha_1-1}z_2
                       \cr \frac{-\alpha_1}{\alpha_1+1} z_1^{\alpha_1}
                             \end{bmatrix}. \]
Now, for any $\alpha_1 \geq 2$, we consider $\beta= (\alpha_1-1,0)$ and we have
\begin{align*}
\Theta_{V,\beta} \begin{bmatrix}
                           \frac{-\alpha_1}{\alpha_1+1} z_1^{\alpha_1-1}z_2
                       \cr \frac{-\alpha_1}{\alpha_1+1} z_1^{\alpha_1}
                             \end{bmatrix}
 &= \gamma_{\beta - e_1}P_{\mathbb{H}_2}V_1^{*(\alpha_1-2)}
  \big( \frac{-\alpha_1}{\alpha_1+1} z_1^{\alpha_1-1}z_2 \big) \\
 &= \gamma_{\beta - e_1}P_{\mathbb{H}_2} \big( -c_{\alpha_1} z_1z_2 \big)\\
 &= d_{\alpha_1} z_1z_2,
\end{align*} where $d_{\alpha_1}= - \gamma_{\beta - e_1}c_{\alpha_1}$ for some non-zero constant $c_{\alpha_1}$. Hence, $\Theta_{V,\beta} \neq 0$. Moreover, we can conclude that for each $\alpha=(\alpha_1,0)$ with
$\alpha_1 \geq 2$, $\Theta_{V,\beta} \neq 0$ where $\beta=(\alpha_1-1,0)$. In other words, there are infinitely many $\beta$'s for which $\Theta_{V,\beta} \neq 0$, that is, the characteristic function $\Theta_V$ is not a polynomial.

Following the above calculation, it is straightforward to conclude that
$V$ is a pure partial isometry, but it is not unitary equivalent to Drury-Arveson shift as its characteristic function is not the zero polynomial. By [Corollary 3.10, \cite{EL18}], it follows that the tuple $V=(V_1,V_2)$ is not regular in the sense of Definition  \ref{regularity definition}.

\vspace{0.1in} \noindent\textbf{Acknowledgment:} The first named author likes to acknowledge Dr. B.K. Das for some fruitful discussions. His research is supported by the institute Post-Doctoral Fellowship of Indian Institute of Technology, Bombay.  The research of the second named author is supported by DST-INSPIRE Faculty Fellowship No. DST/INSPIRE/04/2014/002624, and he is also grateful to Indian Statistical Institute (Bangalore) for the warm hospitality during his visits to Indian Statistical Institute (Bangalore). The research of the third named author is supported in part by NBHM grant NBHM/R.P.64/2014, and the Mathematical Research Impact Centric Support (MATRICS) grant, File No: MTR/2017/000522 and Core Research Grant, File No: CRG/2019/000908, by the Science and Engineering Research Board (SERB), Department of Science \& Technology (DST), Government of India.

\end{document}